\documentclass[10pt,a4paper,reqno]{article} %,draft
\usepackage{mysty}
\usepackage{amssymb}
\usepackage[thinlines]{easytable}

\newcommand{\wtilde}{{\widetilde{\nabla}}}
\usepackage{fullpage}
\usepackage{multirow}
\usepackage{graphicx} 

\SetLabelAlign{parright}{\parbox[t]{\labelwidth}{\raggedleft{#1}}}
\setlist[description]{style=multiline,topsep=4pt,align=parright}%,font=\normalfont

\graphicspath{{./fig/}}

\makeatletter
\let\reftagform@=\tagform@
\def\tagform@#1{\maketag@@@{(\ignorespaces\textcolor{black}{#1}\unskip\@@italiccorr)}}
\newcommand{\iref}[1]{\textup{\reftagform@{\tcr{\ref{#1}}}}}
\makeatother

\usepackage[noindentafter]{titlesec}
\usepackage{bold-extra}

\titlespacing\section{0pt}{11pt plus 4pt minus 2pt}{6pt plus 2pt minus 2pt}
\titlespacing\subsection{0pt}{10pt plus 4pt minus 2pt}{4pt plus 2pt minus 2pt}
\titlespacing\subsubsection{0pt}{8pt plus 4pt minus 2pt}{4pt plus 2pt minus 2pt}
\titlespacing\paragraph{0pt}{6pt plus 4pt minus 2pt}{6pt plus 2pt minus 2pt}

\begin{document}

	\setlength{\abovedisplayskip}{4.5pt}
	\setlength{\belowdisplayskip}{4pt}

\title{Stochastic Three-Operator Splitting Algorithms for Nonconvex and Nonsmooth Optimization Arising from FLASH Radiotherapy}
\author{
Fengmiao Bian\thanks{Department of Mathematics, The Hong Kong University of Science and Technology, Hong Kong, China. E-mail: mafmbian@ust.hk} \and
 		Jiulong Liu \thanks{Institute of Computational Mathematics and Scientific/Engineering Computing, Academy of Mathematics and System Sciences, Chinese Academy of Sciences, Beijing, China. E-mail: jiulongliu@lsec.cc.ac.cn} \and
		Xiaoqun Zhang\thanks{School of Mathematical Sciences, MOE-LSC and Institute of Natural Sciences, Shanghai Jiao Tong University, Shanghai, China. E-mail: xqzhang@sjtu.edu.cn}\and
		Hao Gao \thanks{Department of Radiation Oncology, University of Kansas Medical Center, Kansas City, Kansas, USA. E-mail: hgao2@kumcc.edu} \and
		Jianfeng Cai \thanks{Department of Mathematics, The Hong Kong University of Science and Technology, Hong Kong, China and 
 HKUST Shenzhen-Hong Kong Collaborative Innovation Research Institute, Futian, Shenzhen, China. E-mail: jfcai@ust.hk.}
		}
\date{}
\maketitle

\begin{abstract}
Radiation therapy (RT) aims to deliver tumoricidal doses with minimal radiation-induced normal-tissue toxicity. Compared to conventional RT (of conventional dose rate), FLASH-RT (of ultra-high dose rate) can provide additional normal tissue sparing, which however has created a new nonconvex and nonsmooth optimization problem that is highly challenging to solve. In this paper, we propose a stochastic three-operator splitting (STOS) algorithm to address the FLASH optimization problem. We establish the convergence and convergence rates of the STOS algorithm under the nonconvex framework for both unbiased gradient estimators and variance-reduced gradient estimators. These stochastic gradient estimators include the most popular ones, such as SGD, SAGA, SARAH, and SVRG, among others. The effectiveness of the STOS algorithm is validated using FLASH radiotherapy planning for patients. 
\end{abstract}

\begin{keywords}
Nonconvex optimization $\cdot$  stochastic three-operator splitting $\cdot$ unbiased stochastic gradient $\cdot$  variance-reduced stochastic gradient $\cdot$  FLASH radiotherapy
\end{keywords}

\begin{AMS}
{90C06 $\cdot$ 90C15 $\cdot$ 90C26 $\cdot$ 90C90}
\end{AMS}

%%%%%%%%%%%%%%%%%%%%%%%%%%%%%%%%%%%%%%%%%%%%%%%%%%%%%%%%%%%%%%%%
\section{Introduction}
\subsection{Background}
Radiation therapy (RT) has long been a cornerstone of cancer treatment. With the continuous advancement of medical technology, there are various radiation therapy techniques available, such as IMPT and PBS \cite{zhu2010, Gao2020, Wilson2020}. Although these techniques have been widely adopted in the medical field, complete eradication of cancerous tissue still relies on the dose, which is limited by the risk of severe radiation-induced side effects. A new technique called flash radiotherapy (FLASH-RT) has emerged.  FLASH-RT delivers ultra-high dose rates, several orders of magnitude higher than conventional therapy \cite{Hughes2020}. Clinical data \cite{Hughes2020, Favaudon2014} indicates that  FLASH-RT  can achieve better control with fewer side effects. This is primarily because flash radiotherapy irradiates tissues at ultra-high dose rates ($\ge$ 40 Gy/s), potentially reducing the toxicity of normal tissues while retaining tumor-killing effects. Proton irradiation can achieve this ultra-high dose rate, due to the unique depth-dose characteristics of proton beams, allowing the delivery of ultra-high dose rates to deep tissues \cite{Buonanno2019, Beyreuther2019}. Additionally, clinical data \cite{Girdhani2019,Hughes2020} has demonstrated that FLASH proton radiotherapy, due to its superior immune response capabilities, has become an important treatment modality for certain acute and advanced cancers. 

While FLASH proton radiotherapy may become the primary radiotherapy method for certain tumors, current research on FLASH proton radiotherapy planning has mainly focused on optimizing dose distribution. However, each spot has constraints not only on dose distribution but also on dose rate distribution, highlighting the importance of exploring a new treatment optimization method called FLASH proton radiotherapy with dose and dose rate optimization\cite{gao2020simultaneous,lin2021sddro,gao2022simultaneous}. Mathematically, it can be formulated as the following constraint nonconvex optimization:
\begin{equation}\label{SDRRO}
\begin{aligned}
&\min_{x \in \mathbb{R}^{n}} H(x) := \left\| Ax - b \right\|^2, \qquad \textmd{s.t.} \begin{cases} x_i \geq 0,\\ x_i \geq \alpha,~~ \textmd{if}~~x_i >0,\\  Q x \geq u, \end{cases}
\end{aligned}
\end{equation}
where $x$ is the spot weight to be optimized, $b$ is the dose objective, $A$ represents the forward system matrix, $Q$ is a block matrix composed of the forward matrix that only maps $x$ to the dose in the ROI, $u$ is a vector associated with DVH constraints, and $\alpha$ represents the planning monitor unit. Here $H$ exhibits a large-scale summation structure, representing the sum of planning objectives. See section \ref{sec:experiments} for more details. Currently, most of the research on solving \eqref{SDRRO} employs the nonstandard alternating direction method of multipliers (ADMM) algorithm to solve the convex relaxation problem. A classical technique for solving constraint optimization is to transform it into a multiple summation problem via indicator functions. Here, we employ this technique to equivalently transform \eqref{SDRRO} into the following large-scale nonconvex optimization:
\begin{equation}\label{FRT}
\min_{x \in \mathbb{R}^n } H(x)  +\mathcal{I}_{ C }(x) + \mathcal{I}_{D } (x),
\end{equation}
 where $C=\big\{ x \in \mathbb{R}^n | x_i \in \{0\}\bigcup [ \alpha, +\infty), ~\forall~ i=1,2, \dots, n \big\} =  \mathbb{R}_+^n/(0,\alpha)^n $, the set $D= \big\{  x \in \mathbb{R}^n |  Qx \geq u \big\}$. This transformation allows us to apply classical splitting algorithms to solve \eqref{FRT}. The three-operator splitting algorithm \cite{DY} is an efficient algorithm for solving this kind of composite optimization. However, in our case, $H$ has a large-scale summation structure, and computing $\nabla H$ can be computationally expensive, reducing the efficiency of the method. This prompts us to propose a novel stochastic three-operator splitting algorithm and motivates the research presented below.

\subsection{Model and Algorithm}
Driven by the FLASH proton radiotherapy problem \eqref{FRT} and to broaden the applicability of our research, we consider the following general optimization problem:
\begin{equation}\label{model}
\min_{x \in \mathbb{R}^n} F(x) + G(x) + H(x),
\end{equation}
where $F, G$ and $H$ are proper lower semi-continuous functions. We do not assume convexity for $F$, $G$, or $H$ throughout this paper. Corresponding to the FLASH proton radiotherapy problem \eqref{FRT}, we know $H = \left\| Ax - b \right\|^2$ with a large-scale summation structure, while $F$ and $G$ correspond to indicator functions of two constraint sets. Problem \eqref{model} not only corresponds to the FLASH proton radiotherapy problem but also captures many other applications, particularly in signal/image processing and computer vision,  that involve multiple regularization terms. For instance, deep learning-based models in image processing often include both sparsity regularization terms and neural network-based regularization terms \cite{MKNg2023}. Non-negative low-rank matrix decompositions have both low-rank and non-negativity constraints \cite{FMB, GILLIS2014, PNg}. In compressive sensing, a highly significant type of regularization known as L1-L2 sparse regularization allows the compressive sensing problem to be formulated as a three-term composite nonconvex optimization problem \cite{LYan, YLX}. In robust principal component analysis (RPCA), the optimization problem includes constraints not only on the data term but also on the low-rank factor and sparse factor, which can be formulated as a three-term composite minimization \cite{CCW}. Operator splitting is an important method for solving composite optimization problems of type \eqref{model}. Classic algorithms such as Forward-Backward Splitting (FBS) and Douglas-Rachford Splitting (DRS) can be employed when \eqref{model} $H=0$. In \cite{ABS}, Attouch, Bolte, and Svaiter established the convergence of FBS for nonconvex and nonsmooth optimization under the assumption that $\nabla F$ is Lipschitz continuous. In \cite{LP}, Li and Pong applied the DRS algorithm to nonconvex feasibility problems and established its convergence under similar assumptions.

The above splitting algorithms are not suitable for \eqref{model} because they require computing the proximal operator of $G + H$ in each iteration. Computing this operator is much more expensive or even infeasible compared to computing the proximal operators of $G$ and $H$ separately. To address this limitation, Davis and Yin \cite{DY} proposed a three-operator splitting algorithm for solving \eqref{model}. In the three-operator splitting algorithm, each sub-problem involves solving the proximal operator of a single function, either $F$ or $G$. When the proximal operators of $F$ and $G$ are easy to solve, the three-operator splitting algorithm becomes a simple and effective method. Davis and Yin \cite{DY} established the convergence of a three-operator splitting algorithm for convex optimization. Liu and Yin in \cite{LY} proved the convergence of the three-operator splitting method when $F$ is weakly convex, and $G$ and $H$ are twice continuously differentiable with bounded Hessians. In \cite{BZ} Bian and Zhang established the global convergence of a three-operator splitting algorithm for solving nonconvex optimization problems and provided local convergence rates.

 However, when solving large-scale problems, the TOS algorithm may suffer from computational inefficiency due to the time-consuming computation of $\nabla H$, thereby reducing the efficiency of the TOS method.  To overcome this challenge, we propose a stochastic version of the TOS algorithm by replacing $\nabla H$ with its stochastic approximation. This will lead to a new stochastic three-operator splitting (STOS) algorithm that allows for more efficient computations. We present STOS in Algorithm \ref{alg:STOS} below.

\begin{center}
\begin{minipage}{0.95\linewidth}
	
\begin{algorithm}[H]
\caption{A Stochastic Three-Operator Splitting (STOS)}
\label{alg:STOS}
\begin{algorithmic}
\STATE{{\bf{Step 0.}}  Choose a step-size $\gamma >0$ and an initial point $x_{0}$.}

\STATE{{\bf{Step 1.}}  Set
\begin{subequations}\label{alg1}
\begin{align}
&y_{t+1} \in  \arg\min_{y}  \bigg\{  F(y) + \frac{1}{2\gamma}\|y-x_{t}\|^{2} \bigg\}, \label{algy}\\
&z_{t+1} \in  \arg\min_{z} \bigg\{ G(z) + \frac{1}{2\gamma}\|z - ( 2y_{t+1} - \gamma \wtilde H(y_{t+1}) - x_{t} ) \|^{2} \bigg\},  \label{algz} \\
&x_{t+1} = x_{t} + (z_{t+1} - y_{t+1}). \label{algx}
\end{align}
\end{subequations}
{\bf \qquad\quad~  end while.}
}
\STATE{{\bf{Step 2.}}  If a termination criterion is not met, go to {\bf Step 1}.}

\end{algorithmic}
\end{algorithm}

\end{minipage}
\end{center}
Note that the STOS algorithm is also an extension of these classical algorithms. When $F=0$, Algorithm \ref{alg:STOS} becomes the stochastic Forward-Backward Splitting (FBS) algorithm. When $H=0$, the Algorithm \ref{alg:STOS} corresponds to the classical Douglas-Rachford Splitting (DRS) algorithm. When $F=0$ and $G=0$, this algorithm becomes the stochastic gradient descent algorithm. Here, $\wtilde H$ can be either unbiased gradient estimators (such as SGD, etc.) or variance-reduced gradient estimators (such as SAGA, SVRG, SARAH, etc.).

\subsection{Related work} 

Stochastic splitting algorithms for solving large-scale nonconvex optimization problems have received extensive interest in the last two decades.

When $F=G=0$ in \eqref{model}, the classical stochastic gradient descent algorithm \cite{RM} can solve this kind of problem. In \cite{GL}, Ghadimi and Lan proposed the randomized stochastic gradient (RSG) method to solve \eqref{model} with $F=G=0$ and presented the first non-asymptotic convergence analysis of SGD. They demonstrated that under the assumptions of $H$ being smooth, $\wtilde H$ being an unbiased and bounded variance estimator, the RSG method achieves a solution $x_*$ satisfying $\mathbb{E}[\| \nabla H(x_*) \|^2] \leq \varepsilon$ with an iterative complexity $\mathcal{O}(1/\varepsilon^2)$. Further, the stochastic gradient descent algorithm has also been combined with certain variance-reduced gradient estimators. For instance, Shamir \cite{Shamir} proposed nonconvex SVRG and considered the problem of computing several leading eigenvectors. Allen-zhu and Hazan \cite{ZH} also proposed SVRG, showing its linear convergence rate. Later, Reddi et al. \cite{RHSPS} introduced the nonconvex SAGA method and demonstrated its linear convergence to the global optimal solution for a class of nonconvex optimization.

When $F=0$ in \eqref{model}, this problem can be solved by the classical proximal gradient descent (PGD) algorithm. Ghadimi, Lan, and Zhang \cite{GLZ} introduced a randomized stochastic projected gradient (RSPG) algorithm to solve \eqref{model} with a constraint set $X$ when $F=0$. It is worth noting that when the constraint set $X=\mathbb{R}^n$, the RSPG algorithm becomes the stochastic PGD algorithm. Further, the PGD algorithm combined with variance-reduced gradient estimators has also received a lot of attention.  Reddi et al. \cite{RSPS} integrate the proximal gradient descent with variance-reduced estimators SAGA and SVRG, referred to as PROXSAGA and PROXSVRG, respectively. They showed that under the assumptions of $G$ being nonsmooth convex and $H$ being smooth nonconvex, achieving an $\varepsilon$-accurate solution requires calling the number of proximal oracles to be $\mathcal{O}(1/\varepsilon)$.  Pham et al. \cite{PNPD} combined proximal gradient descent with the SARAH stochastic estimator, referred to as ProxSARAH. Under the assumptions of $H$ being smooth nonconvex and $G$ nonconvex nonsmooth, they demonstrated that ProxSARAH finds an $\varepsilon$-stationary point with complexity $\mathcal{O}(N + N^{1/2}\varepsilon^{-2})$ for $H$ having finite-sum structure.

The study of \eqref{model} with three-term composite structures is limited in the nonconvex stochastic setting. Metel and Takeda \cite{MT} proposed the stochastic proximal gradient method for solving \eqref{model} and showed that the iteration complexity is $\mathcal{O}(1/\varepsilon^3)$ when the distance to the subdifferential mapping of loss is less than $\varepsilon$. Yurtsever, Mangalick, and Sra \cite{YMS} proposed a stochastic version of a three-operator splitting algorithm to solve \eqref{model} by combining unbiased and bounded-variance stochastic gradient estimators, and showed that finding an $\varepsilon$-stationary point requires an iteration complexity of $\mathcal{O}(1/\varepsilon^3)$ based on a variational inequality. Driggs et al. \cite{DTLDS} proposed a stochastic proximal alternating linearization algorithm when $H$ is a cross-term in \eqref{model}, which combines a class of variance-reduced stochastic gradient estimators. The algorithm finds a $\varepsilon$-stationary point with a complexity of  $\mathcal{O}(1/\varepsilon)$. Further, for a class of large-scale three-block composite optimization problems in deep learning, Bian, Liu, and Zhang \cite{BLZ2022} propose a new stochastic three-block splitting algorithm for \eqref{model} with a cross-term.

\subsection{Contributions}
In this paper, we propose a stochastic three-operator splitting (STOS) algorithm to solve the nonconvex and nonsmooth problem \eqref{model}. The main contributions of our paper can be summarized as follows:\\

\begin{itemize}
\item[1.]  We propose a stochastic three-operator splitting (STOS) algorithm for solving \eqref{model}, which combines two types of stochastic gradient estimators: unbiased estimators and variance-reduced estimators. Compared to the stochastic TOS in \cite{YMS} which only considers unbiased gradient estimator, our STOS (Algorithm \ref{alg:STOS}) combines both unbiased gradient estimators and variance-reduced gradient estimators. These include the most commonly used gradient estimators. Hence, our algorithm can select the appropriate gradient estimators for different applications.\\

\item[2.] When the unbiased gradient estimator is incorporated, we establish the convergence and obtain an $\mathcal{O}(1/\varepsilon)$ convergence rate for the STOS algorithm under the variance-bounded assumption. When a variance-reduced gradient estimator is incorporated, we construct a stability function for the STOS algorithm. Furthermore, if the objective function is semi-algebraic, we obtain the convergence and convergence rate of STOS. Compared to the stochastic TOS in \cite{YMS}, our theoretical analysis does not require the convexity of $F$ and $G$.\\

\item[3.] In the numerical experiments, we apply the STOS algorithm to solve the FLASH proton radiotherapy with dose and dose rate optimization in radiation therapy. In all current research work, nonstandard ADMM algorithms are used to solve the convex relaxation of this problem. In this paper, we directly solve this nonconvex problem using the stochastic three-operator splitting algorithm. To the best of our knowledge, this is the first time a stochastic algorithm has been adopted to solve the FLASH proton radiotherapy optimization. All test results demonstrate that STOS significantly outperforms the nonstandard ADMM algorithm.\\
\end{itemize}

\noindent{\bf Notations and paper organization.} Before presenting the main content of the paper, we list some notations used throughout. We denote $\mathbb{R}^n$ as the n-dimensional Euclidean space and $\langle \cdot, \cdot \rangle$ as the inner product.  The norm induced by the inner product is denoted as $\| \cdot \| = \sqrt{\langle \cdot, \cdot \rangle}$. An extended-real-valued function $f: \mathbb{R}^n \to (-\infty,  \infty]$ is proper if it is never $-\infty$ and its domain, dom$f:= \{x \in \mathbb{R}^n: f(x) < + \infty\}$ is nonempty. The function is closed if it is proper and lower semicontinuous. A point $x_*$ is a stationary point of a function $f$ if $0 \in \partial f(x_*)$. $x_*$ is a critical point of $f$ if $f$ is differentiable at $x_*$ and $\nabla f(x_*) = 0$. A function is coercive if $\lim_{\| x \| \to \infty} f(x) = +\infty$. We call $f$ a $\sigma$-convex function if $f - \frac{\sigma \| \cdot \|}{2}$ is a convex function. When $\sigma > 0$, $f$ is a strongly convex function. When $\sigma < 0$, $f$ is a weakly convex function. When $\sigma =0$, $f$ is a convex function.

The rest of this paper is organized as follows. We first present some notation and preliminaries in Section \ref{sec:notation}. The principal theoretical analysis of our proposed algorithm can be found in Section \ref{sec:convergence}, followed by numerical results in Section \ref{sec:experiments}. In Section \ref{sec:conclude}, we give some concluding remarks.

%%%%%%%%%%%%%%%%%%%%%%%%%%%%%%%%%%%%%%%%%%%%%%%%%%%%%%%%%%%%%%%%
%%%%%%%%%%%%%%%%%%%%%%%%%%%%%%%%%%%%%%%%%%%%%%%%%%%%%%%%%%%%%%%%

\section{Notation and preliminaries}
\label{sec:notation}
This section provides the necessary notations and definitions for our paper.  For STOS (Algorithm \ref{alg:STOS}), two types of gradient estimators can be used: unbiased gradient estimators and variance-reduced gradient estimators. The definitions for these two types of gradient estimators are presented below.\\

\subsection{Unbiased Gradient Estimator} 
A class of gradient estimators that are often used in stochastic algorithms are unbiased gradient estimators, such as the most classical stochastic gradient descent (SGD) estimator \cite{RS}. In this paper, we also analyze STOS (Algorithm \ref{alg:STOS}) combined unbiased gradient estimators. We define an unbiased gradient estimator.

\begin{definition}[Unbiased Gradient Estimator]\label{unbiased}
The stochastic gradient estimator $\wtilde H(x)$ is unbiased if
\begin{equation*}
\mathbb{E} [\wtilde H(x)] = \nabla H(x).
\end{equation*}
\end{definition}

\begin{remark}
In many applications, the function $H$ has a finite-sum structure, such as in empirical risk minimization problems encountered in deep learning. In this case, randomness mainly arises from the way samples are drawn from the sum, and the unbiased gradient estimator can be expressed in the following unified form:
\begin{equation}\label{eq500}
\wtilde H(x) = \frac{1}{N}\sum_{i=1}^N  v_i \nabla H_i (x),
\end{equation}
where $\{ v_i \}_{i=1}^N $ is a sampling vector drawn from a certain distribution $\mathcal{D}$. The random variable $v_i$ can take different forms for different sampling methods. A representative sampling method is $b$-nice sampling without replacement, which forms the stochastic gradient descent (SGD) estimator commonly used in deep learning.
\end{remark}
We provide the specific form of SGD, there exist other sampling methods that can produce an unbiased stochastic gradient, such as sampling with replacement and independent sampling without replacement. For more details, we refer to readers to \cite{KR}.
\begin{definition}[SGD \cite{RS}]\label{SGD}
The SGD estimator $\wtilde^{SGD}H(x)$ is defined as follows: 
\begin{equation*}
\wtilde^{SGD} H(x_t) 
= \frac{1}{b} \sum_{j \in J_t}  \nabla H_{j} ( x_t ),
\end{equation*}
where $J_t$ is a random subset uniformly drawn from $\{ 1, \cdots, N \}$ of fixed batch size $b$.
\end{definition}

\begin{remark}
For the $b$-nice sampling without replacement, let $v_i = \frac{1_{i \in J_t}}{p_i}$ with $1_{i \in J_t} = 1$ for $i \in J_t$ and $0$ otherwise, and $p_i = \frac{b}{N}$ in \eqref{eq500}, resulting in SGD estimator.
\end{remark}

\subsection{Variance-Reduced Gradient Estimator.} 
Variance-reduced gradient estimators have been proposed and widely adopted in stochastic algorithms, such as SVRG (stochastic variance reduced gradient) \cite{JZ},  SAG (stochastic average gradient) \cite{SRB}, and SDCA (stochastic dual coordinate ascent) \cite{SZ}. In \cite{DTLDS}, a unified definition is provided for a class of variance-reduced gradient estimators. In this paper, we analyze the combination of STOS (Algorithm \ref{alg:STOS}) with such variance-reduced gradient estimators taken from\cite{DTLDS}. For convenience, we recall the definition of variance-reduced gradient estimators below.

\begin{definition}[Variance-reduced Gradient Estimator {\cite[Definition 2.1]{DTLDS}}]\label{var-redu}
A gradient estimator $\wtilde$ is called variance-reduced if there exist constants $V_1,~V_2,~V_{\Upsilon} \geq 0$ and $\rho \in (0, 1]$ such that
\begin{itemize}
\item[1.] (MSE Bound). There exists a sequence having random variables $\{ \Upsilon_t \}_{t \geq 1}$ of the form $\Upsilon_t = \sum_{i=1}^s \| v_t^i \|^2$ for some random vectors $v_t^i$ such that
\begin{equation}\nonumber
\mathbb{E}_t \| \wtilde H(x_t) - \nabla H(x_t) \|^2 \leq \Upsilon_t + V_1 (\mathbb{E}_t \| x_{t+1} - x_t \|^2 + \| x_t - x_{t-1} \|^2 ),
\end{equation}
and, with $\Gamma_t = \sum_{i=1}^s \| v_t^i \|,$
\begin{equation}\label{var-mse1}
\mathbb{E}_t \| \wtilde  H(x_t) - \nabla H(x_t) \| \leq \Gamma_t + V_2 (\mathbb{E}_t \| x_{t+1} - x_t \| + \| x_t - x_{t-1} \|).
\end{equation}
\item[2.] (Geometric Decay). The sequence $\{ \Upsilon_t \}_{t\geq 1}$ decays geometrically:
\begin{equation}\label{var-geo}
\mathbb{E}_t \Upsilon_{t+1} \leq (1 - \rho) \Upsilon_t + V_{\Upsilon} \Pa{ \mathbb{E}_t \| x_{t+1} - x_t \|^2 + \| x_t - x_{t-1} \|^2 }.
\end{equation}
\item[3.] (Convergence of Estimator). For all sequences $\{ x_t \}_{t = 0}^{\infty}$ satisfying $\lim_{t \to \infty} \mathbb{E} \| x_t - x_{t-1} \|^2$ $ \to 0,$ it follows that $\mathbb{E} \Upsilon_t \to 0$ and $\mathbb{E} \Gamma_t \to 0$.
\end{itemize}
\end{definition}

The variance-reduced gradient estimator defined in Definition \ref{var-redu} is widely used in many algorithms, such as SAGA \cite{DBL}, SARAH \cite{NLST}, SAG \cite{SRB}, and SVRG \cite{KLRT, JZ}. It has been verified in \cite{DTLDS, BLZ} that these four estimators are variance-reduced gradient estimators of this type. Below, we provide the definitions for these four estimators.

\begin{definition}[SAGA \cite{DBL}]\label{SAGA}
The SAGA gradient approximation $\wtilde^{SAGA} H(x)$ is defined as follows:
\begin{equation}\nonumber
\wtilde^{SAGA} H(x_t) 
= \frac{1}{b} \big( \sum_{j \in J_t} \nabla H_j (x_t) - \nabla H_j (\varphi_t^j) \big) + \frac{1}{N} \sum_{i = 1}^N \nabla H_i (\varphi_t^i),
\end{equation}
where $J_t$ is mini-batches containing $b$ indices. The variables $\varphi_t^i$ follow the update rules $\varphi_{t+1}^i = x_t$ if $i \in J_t$ and $\varphi_{t+1}^i = \varphi_t^i$ otherwise. 
\end{definition}

\begin{definition}[SARAH \cite{NLST}]\label{sarah}
The SARAH estimator reads for $t=0$ as 
\begin{equation}\nonumber
\wtilde^{SARAH} H(x_0) = \nabla H(x_0).
\end{equation}
For $t = 1, 2, \dots $, define random variables $p_t \in \{ 0, 1 \}$ with $P(p_t = 0) = \frac{1}{p}$ and $P(p_t = 1) = 1 - \frac{1}{p}$, where $p \in (1, \infty)$ is a fixed chosen parameter. Let $J_t$ be a random subset uniformly drawn from $\{ 1, \dots, N \}$ of fixed batch size $b$. Then for $t = 1, 2, \dots$ the SARAH gradient approximation reads as 
\begin{equation}\nonumber
\wtilde^{SARAH} H(x_t) 
= 
\left\{
\begin{aligned}
\nabla H(x_t) &: \textrm{if}~~p_t = 0,\\
\frac{1}{b} \big( \sum_{j \in J_t} \nabla H_j (x_t) - \nabla H_j (x_{t-1}) \big) + \wtilde^{SARAH} H(x_{t-1}) &: \textrm{if}~~p_t = 1.
\end{aligned}
\right.
\end{equation}
\end{definition}

\begin{definition}[SAG \cite{SRB}]\label{SAG}
The SAG gradient approximation $\wtilde^{SAG}H(x)$ is defined as follows:
\begin{equation*}
\wtilde^{SAG} H(x_t) 
= \frac{1}{N} \sum_{j \in J_t} \big( \nabla H_{j_t} (x_t) - \nabla H_{j_t} (\varphi_t^j) \big) + \frac{1}{N} \sum_{i = 1}^N \nabla H_i (\varphi_t^i),
\end{equation*}
where $J_t$ is mini-batches containing $b$ indices. And the gradient history update follows : $\varphi_{t+1}^{j_t} = x_t$ and $\nabla H_i (\varphi_{t+1}^i) = \begin{cases} \nabla H_i (x_t) ~~~\textmd{if}~i \in J_t,\\ \nabla H_i(\varphi_t^i)~~~\textmd{o.w.} \end{cases}$
\end{definition}

\begin{definition}[SVRG \cite{JZ, KLRT}]\label{SVRG}
The SVRG gradient approximation $\wtilde^{SVRG}H(x)$ is defined as follows:
\begin{equation}\nonumber
\wtilde^{SVRG} H(x_t) 
= \frac{1}{b} \sum_{j \in J_t} \big(  \nabla H_{j_t} (x_t) - \nabla H_{j_t} (\varphi_s) \big) + \frac{1}{N} \sum_{i = 1}^N \nabla H_i (\varphi_s),
\end{equation}
where $J_t$ is a random subset uniformly drawn from $\{ 1, \cdots, N\}$ of fixed batch size $b$. And $\varphi_s$ is a point updated every $m$ step where the $m$ is the number of steps in the inner loop.
\end{definition}

\begin{remark}
Both types of estimators can be combined with our algorithm, but they have their strengths and weaknesses depending on the application. SGD is the simplest stochastic gradient estimator that only relies on the gradient computed on each batch without additional gradient storage or computation. However, SGD may suffer from relatively large variance, which requires a smaller step size to ensure convergence, leading to slow convergence rates. To address this shortcoming, several variance-reduced gradient estimators have been proposed, such as SVRG, SAGA, SAG, and SARAH. These gradient estimators reduce the variance, allowing for larger step sizes and faster convergence rates. However, the design of these stochastic gradient estimators often requires additional gradients (or dual variable) storage and full gradient computation, increasing the amount of storage and computation needed for faster convergence. This can make these variance-reduced gradient estimators difficult to use for some more complex problems, such as structured prediction and neural network learning.  Therefore, for different problems, one must weigh the trade-offs between using SGD for fast computation per iteration with slow convergence or variance-reduced gradient estimators for slower computation per iteration but faster convergence. Regarding the variance-reduced gradient estimators in \cite{DTLDS}, SAG, SVRG, SAGA, and SARAH all belong to this type of estimator, but not all of them are biased gradient estimators. For example, SAGA is unbiased, and SARAH is biased. 
\end{remark}

\subsection{Kurdyka-\L ojasiewicz property} 
The Kurdyka-\L ojasiewicz (KL) property plays a crucial role in the convergence analysis of nonconvex optimization problems. We first recall the definition of KL property, which is particularly applicable to semi-algebraic functions. For a more detailed account, we refer readers to \cite{ABRS, ABS, BDL, BDLS, BDLM} and the references therein. Let $\varepsilon_1$ and $\varepsilon_2$ be two real numbers satisfying $-\inf < \varepsilon_1 < \varepsilon_2 < +\inf$. We define the set $[ \varepsilon_1< F < \varepsilon_2 ] \overset{\mathrm{def}}{=} \{ x\in \mathbb{R}^{m_1} : \varepsilon_1 < F(x) < \varepsilon_2\}$.

\begin{definition}[KL property]\label{KL1}
A function $F: \mathbb{R}^{n} \to \mathbb{R} \cup \{+\infty\}$ has the Kurdyka-\L ojasiewicz property at $x^{*} \in$ dom $\partial F$ if there exist $\eta \in (0, \infty]$, a neighborhood $U$ of $x^{*}$, and a continuous concave function $\varphi : [0, \eta) \to \mathbb{R}_{+}$ such that:

\begin{itemize}
\item [(i)]  $\varphi(0) = 0,~\varphi \in C^{1}((0, \eta))$, and  $\varphi^{'}(s) > 0$ for all $s \in (0, \eta)$;

\item[(ii)] for all $x \in U \cap [F(x^{*}) < F < F(x^{*})+ \eta]$ the Kurdyka-\L ojasiewicz inequality holds, i.e.,
$$ 
\varphi^{'}(F(x) - F(x^{*})) \textmd{dist}(0, \partial F(x)) \geq 1.
$$
\end{itemize}
If $F$ satisfies the Kurdyka-\L ojasiewicz property at each point of dom $\partial F$, then it is called a KL function.
\end{definition}

Roughly speaking, the KL property allows for sharp optimization up to reparameterization via $\varphi$, which is known as a desingularizing function for $F$. A typical class of KL functions are semi-algebraic functions, such as the $\ell_0$ pseudo-norm and the rank function. More specifically, we can find a detailed result on the KL property in \cite[Section 4.3]{ABRS}. Further arguments can be found in \cite[Section 2]{BDL} and \cite[Corollary 16]{BDLS}. Semi-algebraic functions are easier to identify and encompass a wide range of possibly nonconvex functions that arise in applications, as shown in \cite{ABRS,ABS,BDL,BDLS, BDLM}. To clarify, we revisit the definition of semi-algebraic functions.

\begin{definition}[Semi-algebraic set]\label{Sem-201} A semi-algebraic set $S \subseteq \mathbb{R}^{n}$ is a finite union of sets of the form
\begin{equation*}
\Big\{ x \in \mathbb{R}^{n} : h_{1} (x) = \cdots h_{k} (x) = 0, ~g_{1} (x) < 0, \dots , g_{l} (x) < 0 \Big\},
\end{equation*}
where $g_{1}, \dots, g_{l}$ and $h_{1}, \dots, h_{k}$ are real polynomials. 
\end{definition} 
\begin{definition}[Semi-algebraic function]\label{Sem-202} A function $f : \mathbb{R}^{n} \to \mathbb{R}$ is semi-algebraic if its graph $\big\{ (x,~f(x)) \in \mathbb{R}^{n+1} : x \in \mathbb{R}^{n} \big\}$ is semi-algebraic.
\end{definition} 

\begin{lemma}[KL inequality in the semi-algebraic cases]\label{klsemi}
Let $h$ be a proper closed semi-algebraic function on $\mathbb{R}^n$. Then, $h$ satisfies the Kurdyka-\L ojasiewicz property at all points in dom $\partial h$ with $\varphi(s) = c s^{1- \theta}$ for some $\theta \in [0,1)$ and $c > 0.$
\end{lemma}

%%%%%%%%%%%%%%%%%%%%%%%%%%%%%%%%%%%%%%%%%%%%%%%%%%%%%%%%%%%%%%%%
%%%%%%%%%%%%%%%%%%%%%%%%%%%%%%%%%%%%%%%%%%%%%%%%%%%%%%%%%%%%%%%%

\section{Convergence analysis}
\label{sec:convergence}
This section focuses on the convergence analysis of STOS (Algorithm \ref{alg:STOS}) combined with two different types of stochastic gradient estimators: unbiased estimator and variance-reduced estimator. Our analysis is based on the following assumptions concerning the objective functions $F, G$, and $H$.

\begin{assumption}\label{ass-fun}
Functions $F$, $G$, and $H$ satisfy the following:
\begin{itemize}
\item[(a1)] $F$ has a Lipschitz continuous gradient, i.e., there exists a constant $L>0$ such that
\begin{equation*}
\| \nabla F(y_1) - \nabla F(y_2) \| \leq L \| y_1 - y_2 \|  \quad \forall ~ y_1, y_2 \in \mathbb{R}^n.
\end{equation*}
\item[(a2)] $G$ is a proper closed function with a nonempty mapping $\mathcal{P}_{\gamma G} (x)$ for any $x$ and for $\gamma >0$.\\
\item[(a3)] For each $i = 1, 2, \cdots, N$, $H_i$ has a Lipschitz continuous gradient, i.e., there exists a constant $\beta > 0$ such that
\begin{equation*}
\| \nabla H_i (y_1) - \nabla H_i (y_2) \| \leq \beta \| y_1 - y_2 \|   \quad \forall~~y_1, y_2 \in \mathbb{R}^n.
\end{equation*}
\end{itemize}
\end{assumption}
If $F$ has a Lipschitz continuous gradient, then we can always find $l \in \mathbb{R}$ such that $F + \frac{l}{2} \| \cdot \|^2$ is convex, in particular, $l$ can be taken to be $L$. We also present the optimality conditions for Algorithm \ref{alg:STOS}:  
\begin{subequations}\label{opti}
\begin{align}
& 0 = \nabla F(y_{t+1}) + \frac{1}{\gamma} (  y_{t+1} - x_{t}), \label{opti-1}\\
& 0 \in \partial G(z_{t+1}) + \frac{1}{\gamma} (z_{t+1} + \gamma \wtilde H(y_{t+1})  - 2y_{t+1} + x_{t}). \label{opti-2}
\end{align}
\end{subequations}
These conditions will be frequently used in the convergence analysis.

\subsection{Convergence rates for unbiased gradient estimator}
\label{conv:unbiased}
In this subsection, we consider the stochastic nonconvex problem \eqref{model}  with $H$ being the expectation of a function of a random variable, i.e., $H(x) = \mathbb{E}_{\xi} H(x, \xi)$, where $\xi$ is a random variable with distribution $\mathcal{P}$:
\begin{equation}\label{unmodel}
\min_{x \in \mathbb{R}^n}  F(x) + G(x) + \mathbb{E}_{\xi} H(x, \xi).
\end{equation}
As a special case of \eqref{unmodel}, if $\xi$ is a uniformly random vector defined on a finite support set $\Omega:= \{\xi_1, \xi_2, \cdots, \xi_N\}$, then \eqref{unmodel} simplifies to the following  finite-sum minimization problem:
\begin{equation*}
\min_{x \in \mathbb{R}^n} F(x) + G(x) + \frac{1}{N}\sum_{i = 1}^N H_i (x),
\end{equation*}
where $H_i (x):= H(x, \xi_i)$ for $i= 1, \cdots, N$. Here, the stochastic gradient estimator $\wtilde H$ in Algorithm \ref{alg:STOS} adopts the following SGD estimator: 
\begin{equation*}
\wtilde H(x) = \frac{1}{b} \sum_{\xi \in J_t}  \nabla H(x, \xi),
\end{equation*}
where $J_t$ is a set of $b$ $i.i.d.$ samples from distribution $\mathcal{P}$. Our convergence analysis relies on the following assumptions.

\begin{assumption}\label{ass-bv}
We assume that
\begin{itemize}
\item[(i)] $\nabla H(x, \xi)$ is an unbiased  estimator of $\nabla H(x)$, i.e.,
\begin{equation*}
\mathbb{E}_{\xi} [\nabla H(x, \xi)] = \nabla H(x) \qquad  \forall x \in \mathbb{R}^n.
\end{equation*}
\item[(ii)] $\nabla H(x, \xi)$ has bounded variance, i.e.,
\begin{equation*}
\mathbb{E}_{\xi} [\| \nabla H(x, \xi) - \nabla H(x) \|^2 ] \leq \sigma^2 \qquad  \forall x \in \mathbb{R}^n.
\end{equation*}
\end{itemize}
\end{assumption}

\begin{remark}
Unbiased gradient estimators with bounded variance are often used in convergence analysis of general stochastic algorithms \cite{FLLZ, GLZ, GL, OHTG, MT}. In practical applications, many unbiased gradient estimators satisfy this assumption, and one can refer to the references \cite{GL1, MBPS, MBBF} for more details.
\end{remark}

\begin{definition}
For  $\varepsilon > 0$, we call $x_*$ an $\varepsilon$-stationary point if it satisfies
\begin{equation*}
dist(0, \nabla F(x_*) + \partial G(x_*) + \nabla H(x_*)) \leq \varepsilon,
\end{equation*}
where $\textmd{dist}(0, A)$ denotes the distance between $0$ and the set $A$.
\end{definition}

In the following, we establish the convergence analysis of STOS (Algorithm \ref{alg:STOS}) with unbiased estimators. We show that STOS convergences to an $\varepsilon$-stationary point and obtain its convergence rate. The proof of Theorem \ref{conv-unbiased} is given in Appendix \ref{app1.2}.

\begin{theorem}\label{conv-unbiased}
Suppose that Assumptions \ref{ass-fun} and \ref{ass-bv} hold.  Let  $\{( x_t, y_t, z_t) \}_{t \geq 0}$ be a sequence generated by STOS (Algorithm \ref{alg:STOS}).  Choose $\gamma > 0$ such that
\begin{equation*}
\Lambda (\gamma) := \frac{1 - \gamma (1+ l + 2\beta)}{2\gamma} - \frac{(3\gamma +2)(\gamma L^2 + 2l + 2L)}{2} > 0.
\end{equation*}
Then $z_{\tau}$, randomly returned by the algorithm for $\tau \in \{1, 2, \cdots, T \}$, satisfies
\begin{equation*}
\begin{aligned}
&\mathbb{E}_{\tau} \mathbb{E}~\textmd{dist}^2(0, \partial G(z_{\tau}) + \nabla F(z_{\tau}) + \nabla H(z_{\tau})) \\
&\leq  \frac{1 + \gamma^2(L^2 + \beta^2)}{2\gamma^2 \Lambda(\gamma)} \frac{(1+ \gamma L)^2}{T} \big(\delta_0 + \frac{3\sigma^2}{2}\big) +  \frac{\sigma^2}{2T} =\mathcal{O}(T^{-1}).
\end{aligned}
\end{equation*}
where $L$ and $\beta$ are Lipschitz constants.
\end{theorem}

\begin{remark}
From the expression for $\Lambda(\gamma)$, we know $\Lambda(\gamma) \to +\infty$ as $\gamma \to 0$. Thus, given $l, L$ and $\beta$, it is always true that $\Lambda(\gamma) > 0$ for $\gamma$ small sufficiently. Indeed, we can get a computable thresholding for $\gamma$:
\begin{equation*}
0 < \gamma < \min \Big\{ \frac{1}{L}, \frac{6 L^2 + (2 \beta + 5l + 10)L + 6l}{2L} \Big\}.
\end{equation*}
\end{remark}

\subsection{Convergence rates for variance-reduced gradient estimator}
We have provided the convergence guarantee for STOS (Algorithm \ref{alg:STOS}) when combined with unbiased gradient estimators. However, some variance-reduced gradient estimators, such as SARAH, can make the algorithm more effective for certain large-scale image processing problems \cite{BLZ}. To allow STOS to use a wider variety of stochastic gradient estimators, we further provide the convergence guarantee for STOS combined with the variance-reduced gradient estimator (Definition \ref{var-redu}) and obtain its convergence rate. The convergence analysis is based on the Kurdyka-\L ojasiewicz inequality. For this purpose, it is crucial to construct an energy function that decreases along the sequence generated by Algorithm \ref{alg:STOS}. Thus, we define 
\begin{equation}\label{Phi}
\begin{aligned}
\Phi(x,y,z,y^{\prime}) &=F(y) + G(z) + H(y) + \frac{1}{2\gamma} \| y - x \|^2 -  \frac{1}{2\gamma}\|z - x \|^2  + \langle \nabla H(y), z - y \rangle - \frac{1}{2}\| z - y \|^2 \\
&\quad + C_1 \| y - y^{\prime}\|^2
\end{aligned}
\end{equation}
for $C_1 > 0$. Denote $\Phi_t = \Phi(x_t, y_t, z_t, y_{t-1})$. Next, we give the energy function $\Psi_t$ associated with STOS (Algorithm \ref{alg:STOS}) as follows:
\begin{equation}\label{Psi}
\Psi_t =  \Phi_t - \mathbb{E}_t \| \wtilde H(y_t) - \nabla H(y_t) \|^2  +\frac{5}{2\rho} \Upsilon_t + \frac{5V_{1}\rho + 5 V_{\Upsilon}}{2\rho} \mathbb{E}_t \|y_{t} - y_{t-1} \|^2.
\end{equation}
The decreasing property of $\Psi_t$ in expectation is given in the following theorem.

\begin{theorem}\label{decrease}
Suppose that Assumption \ref{ass-fun} holds.  Let  $\{( x_t, y_t, z_t) \}_{t \geq 0}$ be a sequence generated by STOS (Algorithm \ref{alg:STOS}) using variance-reduced gradient estimators. 
Then for any $t \geq 1$, there holds
\begin{equation*}
\mathbb{E}_t [ \Psi_{t+1}  + C_1\| y_t - y_{t-1} \|^2 + \Lambda_1(\gamma) \| y_{t+1} - y_t \|^2 ] \leq \Psi_t,
\end{equation*}
where $C_1 > 0 $ is a given constant and  
\begin{equation*}
\Lambda_1(\gamma) :=\frac{1- \gamma (1+ l + 2\beta) - 2\gamma C_1}{2\gamma} - \frac{5V_{1}\rho + 5 V_{\Upsilon}}{\rho} -  \frac{3\gamma + 2}{2\gamma} ((-1 + 2\gamma l) + (1+ \gamma L)^2) - (1+\gamma L)^2.
\end{equation*}
Moreover, if $\Lambda_1(\gamma) > 0$, then $\Psi_t$ is a decreasing function and
\begin{equation*}
\sum_{t=0}^{\infty} \| y_{t+1} - y_t \|^2 < \infty, \qquad \sum_{t=0}^{\infty} \| x_{t+1} - x_t \|^2 < \infty,  \qquad \sum_{t=0}^{\infty} \| z_{t+1} - y_{t+1} \|^2 < \infty  \qquad a.s..
\end{equation*}

\end{theorem}
\begin{remark}
We now show that choosing a suitable $\gamma$ always makes $\Lambda_1(\gamma) > 0$. From the expression for $\Lambda_1(\gamma)$ we have that $\Lambda_1(\gamma) \to + \infty$ as $\gamma \to 0$. Thus it is always true that $\Lambda_1(\gamma) >0$ when $\gamma$ is sufficiently small. Here we also provide a computable threshold for $\gamma$. Denote
\begin{equation*}
\mathcal{T}_{\gamma} = \frac{-\big(2K + (6+4L)(\frac{l}{L} +1) +8\big) + \sqrt{\big(2K + (6+4L)(\frac{l}{L} +1) +8\big)^2 + (12L + 8 L^2)}}{6L + 4 L^2},
\end{equation*}
where $K = \frac{1+l+2\beta}{2} + \frac{5V_1\rho + 5 V_{\Upsilon}}{\rho} + C_1$. Then when $0 < \gamma < \min\{\frac{1}{L}, \mathcal{T}_\gamma \}$, we have $\Lambda_1(\gamma) > 0$.
\end{remark}

The proof of Theorem \ref{decrease} is provided in Appendix \ref{app1.3}. Next, we will establish the global convergence in expectation for the whole sequence generated by Algorithm \ref{alg:STOS}. See Appendix \ref{app1.3} for details of the proof.
\begin{theorem}\label{thmKLconver}
Suppose that $F$, $G$, and $H$ are semialgebraic functions with KL exponent $\theta \in [0, 1)$, and the sequence $ \{ (x_t, y_t, z_t)\}$ is generated by Algorithm \ref{alg:STOS}. Then, either $\{(x_t, y_t, z_t)\}$ is a critical point after a finite number of iterations, or $\{ (x_t, y_t, z_t )\}_{t=0}^{\infty}$ almost surely satisfies the finite length property in expectation:
\begin{equation*}
\sum_{t=0}^{\infty} \mathbb{E} \| x_{t+1} - x_t \| < \infty, \quad  \sum_{t=0}^{\infty} \mathbb{E} \|y_{t+1} - y_t \| < \infty ~~\textmd{and}~~\sum_{t=0}^{\infty} \mathbb{E} \| z_{t+1} - z_t \| < \infty.
\end{equation*}
Moreover, there exists an integer $m$ such that for all $i > m$,
\begin{equation*}
\sum_{t=m}^{i} \mathbb{E} \| y_{t+1} - y_t \|^2 + \mathbb{E} \| y_t - y_{t-1} \|^2 \leq \frac{(2\sqrt{2}+1)\sqrt{s}}{3 \big( C+ \frac{2\sqrt{sV_{\Upsilon}}}{\rho}\big)\rho}\sqrt{\mathbb{E}\Upsilon_m} + \frac{(4\sqrt{2} + 2) \big( C+ \frac{2\sqrt{sV_{\Upsilon}}}{\rho}\big)}{3\max\{\Lambda(\gamma),C_1\}} \Delta_{m, i-1},
\end{equation*}
where $C$ and $C_1$ are both positive constants, and $\Delta_{m,n} = \varphi(\mathbb{E} [\Phi_m - \Phi_m^*]) - \varphi(\mathbb{E}[\Phi_n - \Phi_n^*])$.
\end{theorem}

Furthermore, we derive the convergence rates for the sequence $\{ (x_t, y_t, z_t )\}_{t \geq 0}$ generated by Algorithm \ref{alg:STOS} based on the KL exponent. The calculation of the KL exponent is detailed in \cite{LPex}. We omit the proof of convergence rate, which is similar to the proof of Theorem 3.7 in \cite{BLZ}.
\begin{theorem}\label{rate}
Suppose that  $2\tau \geq \beta \|A\|^2$, $0 <\sigma \leq 1$ and Assumption \ref{ass-fun} is satisfied. Moreover, suppose that $F$, $H$ and $G$ are semialgebraic functions with KL exponent $\theta \in [0, 1)$. Let $\{ (x_t, y_t, z_t) \}_{t=0}^{\infty}$ be a bounded sequence generated by STOS (Algorithm \ref{alg:STOS}) using a variance-reduced gradient estimator. Then, the following convergence rates hold almost surely:

\begin{itemize}
\item[1.] If $\theta = 0,$ then there exists an $m \in \mathbb{N}$ such that $\mathbb{E}\Phi(X_t) = \mathbb{E}\Phi(X_*)$ for all $t \geq m;$

\item[2.] If $\theta \in (0, \frac{1}{2}]$, then there exists $d_1 > 0$ and $\tau \in [ 1 -\rho, 1)$ such that $\mathbb{E}\| X_t - X_* \| \leq d_1 \tau^t$;

\item[3.] If $\theta \in (\frac{1}{2}, 1)$, then there exists a constant $d_2 > 0$ such that $\mathbb{E} \| X_t - X_* \| \leq d_2 t^{-\frac{1 -\theta}{2\theta - 1}}$.
\end{itemize}
\end{theorem}

%%%%%%%%%%%%%%%%%%%%%%%%%%%%%%%%%%%%%%%%%%%%%%%%%%%%%%%%%%%%%%%%
%%%%%%%%%%%%%%%%%%%%%%%%%%%%%%%%%%%%%%%%%%%%%%%%%%%%%%%%%%%%%%%%
\section{ Numerical experiments}
\label{sec:experiments} 

In this section, we implement the STOS algorithm on FLASH proton radiotherapy with dose and dose rate constraints. All experiments are run in MATLAB R2021a on a desktop equipped with a 3.9GHz 16-core AMD processor and 64GB memory.

\subsection{FLASH proton radiotherapy with dose and dose rate optimization}
\subsubsection{Problem background}
FLASH proton radiotherapy is an emerging cutting-edge technology that can deliver ultra-high doses of radiation in an extremely short period, within a fraction of a second. Although this new treatment modality is still in the experimental stage, it has demonstrated the potential to better protect normal tissues compared to conventional radiation. In clinical practice, proton beams can be utilized to deliver this high-dose radiation, making FLASH proton radiotherapy an intriguing prospect for future cancer treatment \cite{gao2020simultaneous, lin2021sddro}.

However, the majority of research on FLASH proton radiotherapy primarily focuses on radiation therapy problems with a minimum threshold to the number
of protons delivered per spot, such as the well-known minimization monitor-unit (MMU) problem \cite{zhu2023orthogonal, cai2022minimum}. In recent developments, modeling the FLASH effect while satisfying both dose and dose rate constraints has demonstrated greater effectiveness \cite{gao2022simultaneous}. Therefore, it becomes imperative to tackle the challenge of simultaneous dose and dose rate constraints as well as the MMU in FLASH proton radiotherapy, which can be formulated as the following nonconvex-constrained optimization problem:
\begin{equation}\label{FlashRT}
\begin{aligned}
&\min_{x \in \mathbb{R}^n}~~\frac{1}{2} \big\| A x - y \big\|^2 = \sum_i^{N} \frac{1}{2}\big\| \langle  A_i,  x\rangle -y_i \big\|_2^2\\
&\textmd{s. t.} 
\begin{cases}
x \geq 0,\\
x \geq \alpha, \text{ if } x>0,\\
Q x \geq u.
\end{cases}
\end{aligned}
\end{equation}
Here, $x$ represents proton spot weights to be optimized, and $y$ is the weighted vector of objective constraint values. Let  $u = [\bm{0}_{n\times 1} ; \ {\mu_{dr}}\bm{1}_{n\times 1}]$ where $\bm{0}_{n\times 1} \in \mathbb{R}^{n\times 1}    $ with all elements equal to $0$ . $A$ represents the forward system matrix, a linear operator mapping $x$ to $y$, and $A_i$ denotes the $i$-th row of $A$. $\alpha$ is the planning minimum spot weight threshold, known as the MMU constraint.  The FLASH constraints are enforced in a region of interest (ROI), we use $B \in \mathbb{R}^{ p \times n}$ to denote the partial forward system matrix which only maps $x$ to the dose in the ROI ($y_{roi}$). The ROI, as defined in \cite{gao2022simultaneous} is much smaller than the combined regions of CTV, body, and OARs.  Consequently, the size of $B$ exhibits significantly smaller dimensions compared to the size of $A$. Here, $Q$ can be explicitly written as the following block matrix:
%$\tilde{x} = [x; x]$, which is formed by stacking two $x$ vectors. Also, let  , where $\bm{\mu_{dr}}$ is a vector with all elements equal to $\mu_{dr}$
\begin{equation*}
\begin{bmatrix}
\frac{\alpha}{t}(B \circ B) & B\\ B & {\bf 0}_{p\times n } 
\end{bmatrix}
\begin{bmatrix}
 { I}_{n\times n } \\  { I}_{n\times n } 
\end{bmatrix}.
\end{equation*}
This shows that the constraint $Q {x} \geq u $ is equivalent to two constraints, one regarding dose rate: $(B \circ B) (\frac{\alpha}{t} x) -  \mu_{dr}   B x \geq 0$ with $ \mu_{dr}$ being the lower bound of dose rate in ROI and the other regarding dose: $B x - \mu_{d} \geq 0$ with $ \mu_{d}$ being the lower bound of dose in ROI. Here ``$\circ$'' represents dot product, ``$\geq$'' applies to each element, and $t$ is minimum spot duration. Additional physical interpretation of these two constraints regarding the FLASH effect can be found in \cite{gao2022simultaneous}.

\subsubsection{Model and algorithm}
Based on our knowledge, all of the literature uses an nonstandard ADMM algorithm to solve the convex relaxation of problem \eqref{FlashRT}. The technique of transforming constrained optimization problems into unconstrained ones using indicator functions is a commonly employed approach in the field of optimization. Here, we apply this technique to transform problem \eqref{FlashRT} into the following equivalent unconstrained optimization problem:
\begin{equation}\label{FRT-un}
\min_{x \in \mathbb{R}^n } H(x)  +\mathcal{I}_{C }(x) + \mathcal{I}_{D }(x),
\end{equation}
where we denote the set $C=\big\{ x \in \mathbb{R}^n | x_i \in \{0\}\bigcup [ \alpha, +\infty), ~\forall~ i=1,2, \dots, n \big\} =  \mathbb{R}_+^n/(0,\alpha)^n $, the set $D= \big\{  x \in \mathbb{R}^n | Q x \geq u \big\}$. Here $\mathcal{I}$ represents the indicator function of the set, defined as
\begin{equation*}
\mathcal{I}_{D}(x)= \begin{cases} 0, \qquad ~~~~~\textmd{if}~x \in D,\\ +\infty, \qquad ~\textmd{if}~x \notin D.  \end{cases}
\end{equation*}
We note that problem \eqref{FRT-un} can be solved by the classical three-operator splitting algorithm. Furthermore, since $N$ is typically on the order of millions in FLASH proton therapy optimization, our proposed STOS algorithm can effectively solve \eqref{FRT-un}. 

We apply the STOS method to solve \eqref{FRT-un} with $H(x) = \sum_i^{N} \frac{1}{2}\big\| \langle  A_i,  x\rangle -y_i \big\|_2^2$, $G(x) = \mathcal{I}_{C} (x)$, and $F(x) = \mathcal{I}_{D}(x)$. However, we note that in our convergence analysis, we require the function $F$ to be smooth. Therefore, we approximate the indicator function $\mathcal{I}_{D}(x)$ using the distance function $\textmd{dist}_{D}(x)= \min_{z \in D} \frac{\lambda}{2}\|y - z \|^2$. This is mainly because for the first subproblem in Algorithm \ref{alg:STOS}:
\begin{equation}\label{eq:subprob1}
\arg\min_{y} \Big\{ \textmd{dist}_{D} (y) + \frac{1}{2\gamma} \| y - x_t \|^2 \Big\} = \frac{1}{1 + \gamma \lambda} \Big( x_t + \gamma\lambda \mathcal{P}_{D}(x_t) \Big).
\end{equation}
Here, $\mathcal{P}_{D}$ is the projection operator onto set $D$. It is worth noting that in \eqref{eq:subprob1}, as the $\gamma \lambda \to \infty$, it approaches to $\mathcal{P}_D(x_t)$, which is the solution of the first subproblem in Algorithm \ref{alg:STOS} when $F=\mathcal{I}_D$. Then we can apply STOS  to solve \eqref{FRT-un} with $H(x) = \sum_i^{N} \frac{1}{2}\big\| \langle  A_i,  x\rangle -y_i \big\|_2^2$, $G(x) = \mathcal{I}_{C} (x)$, and $F(x) = \textmd{dist}_{D}(x)$, it yields the following algorithm:
\begin{equation*}
\begin{cases}
y_{t+1} = \arg\min_{y}  \textmd{dist}_{D}(x) + \frac{1}{2\gamma} \| y - x_t \|^2, \\\
z_{t+1} = \mathcal{P}_{C} (2 y_{t+1} - \gamma \wtilde H(y_{t+1}) - x_t),\\
x_{t+1} = x_t +(z_{t+1} - y_{t+1}).
\end{cases}
\end{equation*}
For the first subproblem, due to the smoothness of $\dist_{D}(x)$, we can directly compute the gradient and further utilize a quasi-Newton method in MATLAB to solve it. The second subproblem is easy to solve, and it has an explicit solution:
\begin{equation*}
\big[ \mathcal{P}_{C}(x) \big]_{i} = \begin{cases} x_i, ~~~~ \textmd{if} ~~ x_i \geq \frac{\alpha}{2},\\  0, \quad \textmd{otherwise}. \end{cases}
\end{equation*}
Further, we verify that $F$, $G$, and $H$ satisfy the assumptions of a semi-algebraic function. According to \cite[Section 4.3]{ABRS}, $H(x)$ and $F(x)$ are semi-algebraic functions. From Appendix of \cite{BST}, the set $C$ is a semi-algebraic set, therefore the indicator function of $C$ is also semi-algebraic.

Next, we evaluate the STOS algorithm in two different cases separately, i.e., FLASH proton RT for lung and brain,  and we also compared the nonstandard ADMM in \cite{gao2020simultaneous} in which the nonlinearity of the constraint is linearized, and in all experiments, the minimum spot weight threshold $\alpha$ are fixed as \cite{gao2020simultaneous} suggested (especially 160 for our experiments). Before presenting the results, we shall introduce several quality measures designed to assess the effectiveness and accuracy of therapy planning. One of these key metrics is the Conformal Index, denoted as $(\mathrm{CI})$, which is defined as $V_{100}^2 / \left(V \times V_{100}^{\prime} \right)$, where $V_{100}$ represents the PTV volume receiving at least $0\%$ of the prescription dose, $V$ represents the volume of PTV, and $V^{\prime}_{100}$ represents the total volume receiving at least $100\%$ of the prescription dose. Therefore, the value of $\mathrm{CI}$ ranges from 0 to 1, and ideally, $\mathrm{CI}=1$. $D_{\max}$, $D_{\max}^{cord}$, and $D_{\max}^{lung}$ represent the maximum dose received by CTV(Clinical Target Volume), cord, and lung, respectively. $D_{mean}^{roi}$, $D_{mean}^{esophagus}$, and $D_{mean}^{body}$ represent the average dose received by ROI, esophagus, and the body, respectively. The quantitative dose coverage and dose coverage in percentage are calculated as $\frac{\#\{w_d>0\}}{\#\{w_d\}}\times 100$ and  $\frac{\#\{w_{dr}>0\}}{\#\{w_{dr}\}}\times 100$  providing an assessment of the degree to which constraints related to $w_d$ and $w_{dr}$ are met.

\subsubsection{FLASH proton radiotherapy for the lung}
Radiation therapy is a common and effective treatment method for lung cancer. The first experiment of FLASH proton radiotherapy is conducted for the lung. The dose influence matrix, denoted as $A$, is generated via MatRad \cite{wieser2017development}, with $5mm$ spot width, and $3mm$ lateral spacing on $3mm^3$ dose grid. The beam angles are empirically chosen as $\left(0^{\circ}, 120^{\circ}, 240^{\circ}\right)$. In this experiment, the regions of CTV, the body, and the organs at risk (spinal cord, lung, esophagus) are chosen for dose planning. The size of the dose influence matrix $A$ is $1323239 \times 1505$. Matrix $B$ corresponds to a portion of $A$ on ROI and has a size of $24353 \times 1505$. The data $y$ with the length $N = 1323239$ is shuffled and divided into 8 parts and 16 parts for testing our STOS algorithm. The parameter $\mu_{dr}$ is set to $40$ Gy, while $\mu_d$ is set to $8$ Gy. The parameter $\gamma$ for STOS is set to $2.5 \times 10^4$.

\begin{figure}[htbp]
  \centering
  \subfloat[Varying batch sizes ]{\includegraphics[width=0.4\textwidth]{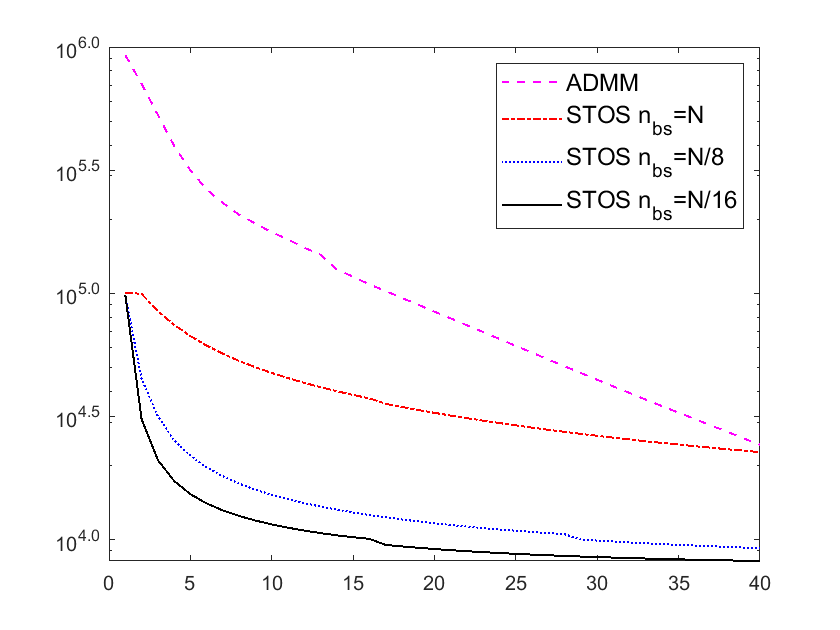}}
  \subfloat[ Varying gradient methods] {\includegraphics[width=0.4\textwidth]{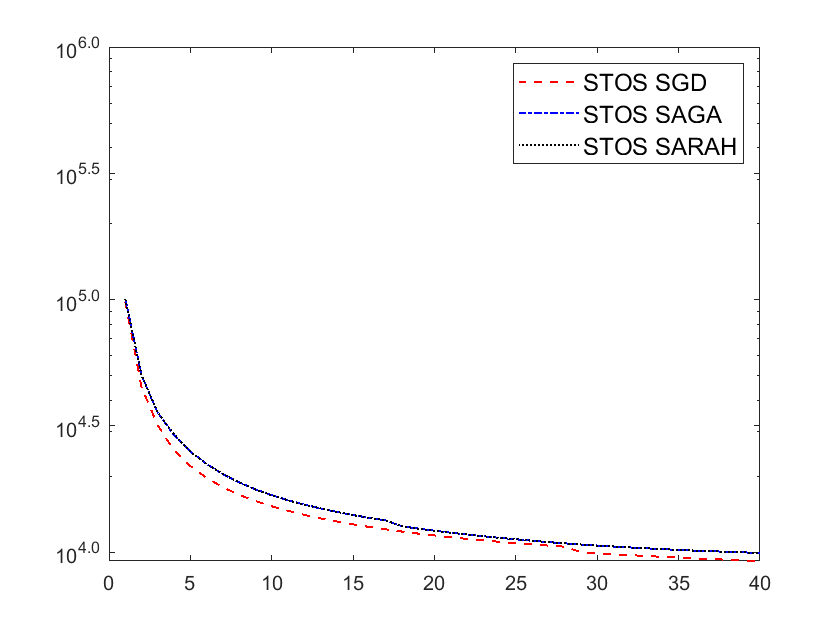}}
  \caption{Test loss for different methods with the same epochs for the lung.}
  \label{fig:energy_lung}
\end{figure}

\begin{table}[htbp]
\centering
\resizebox{\textwidth}{!}{
\begin{tabular}{l|ccccccccccc}
  \hline
    & $D_{max}$ & CI & $D_{mean}^{roi}$ & $D_{max}^{cord}$   & $D_{max}^{lung}$ & $D_{mean}^{esophagus}$ & $D_{mean}^{body}$ & $P_{d}$ & $P_{dr}$  \\ % 
  \hline
ADMM   & 109.62 & 0.56 & 1.29 & 9.85  & 18.54 & 2.57 & 1.14 &  58.80 &90.01  \\
STOS SGD $N$  & 118.51 & 0.67 & 1.19 & 7.28  & 17.49 & 1.95 & 1.00 &  67.76 &92.94\\
STOS SGD $\frac{N}{8}$ & 118.83 & 0.66 & 1.19 & 7.13  & 17.85 & 1.81 & 1.02 & 71.14 &94.16 \\
STOS SGD $\frac{N}{16}$  & 118.82 & 0.65 & 1.19 & 7.14 & 17.86 & 1.81 & 1.02 & 71.63 &94.21\\
STOS SAGA $\frac{N}{8}$ & 109.50 & 0.66 & 1.18 & 8.09 & 16.93 & 1.86 & 1.01 &  71.54 &94.17\\
STOS SARAH $\frac{N}{8}$ & 109.53 & 0.66 & 1.18 & 8.11  & 16.96 & 1.86 & 1.01  & 71.09 &94.31 \\
  \hline
\end{tabular}
}
\caption{\label{tab:dose_lung} Physical dose parameters for Lung (dose unit: Gy ) }
\end{table}

\begin{figure}[htbp]
  \centering
  \subfloat[Varying batch sizes ]
  {\includegraphics[width=0.4\textwidth]{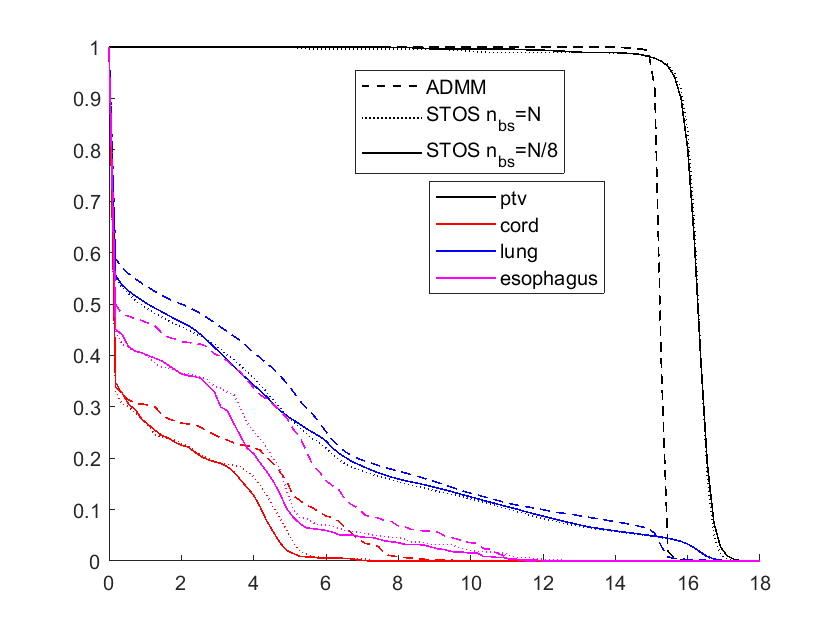}}
 % [b]
  \subfloat[Varying batch sizes ]
  {\includegraphics[width=0.4\textwidth]{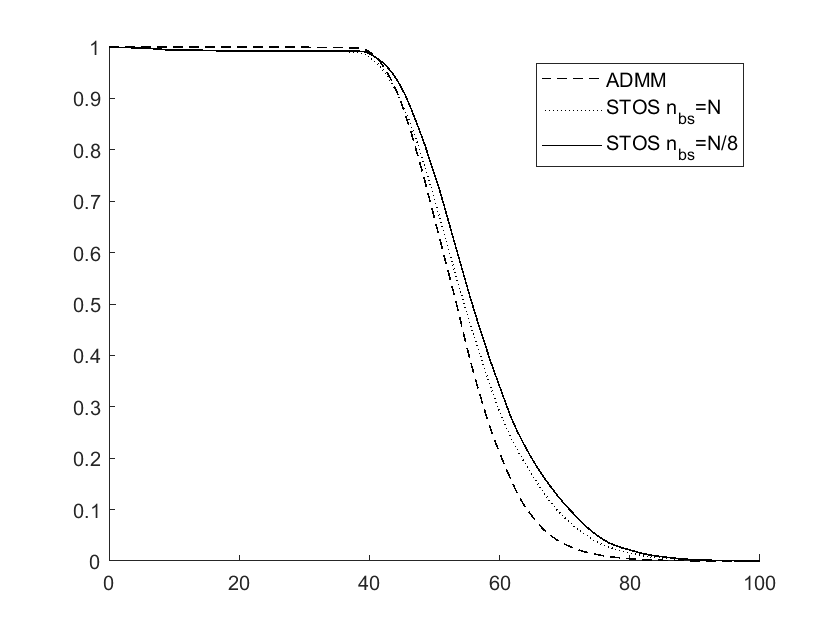}}\\
    \subfloat[Varying  gradient methods]
  {\includegraphics[width=0.4\textwidth]{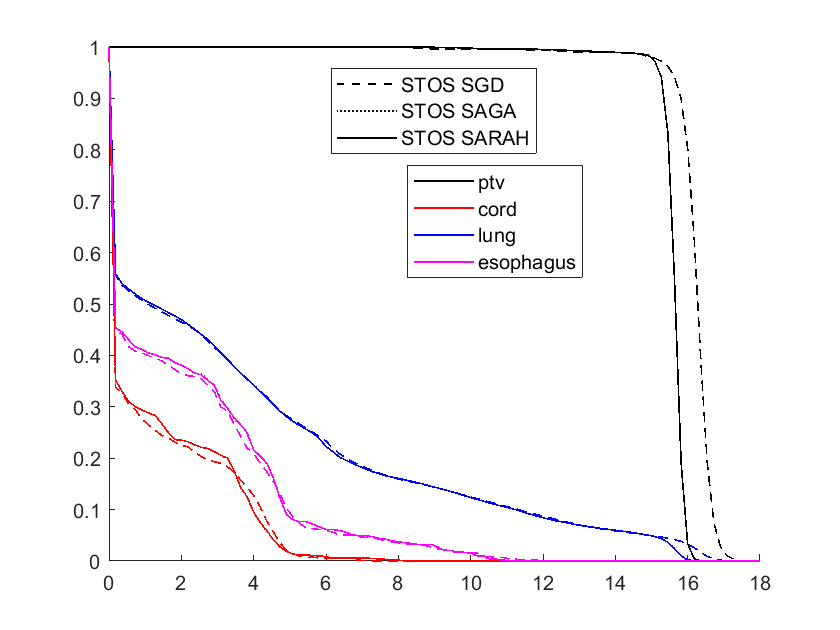}}
 % [b]
  \subfloat[Varying gradient methods]
  {\includegraphics[width=0.4\textwidth]{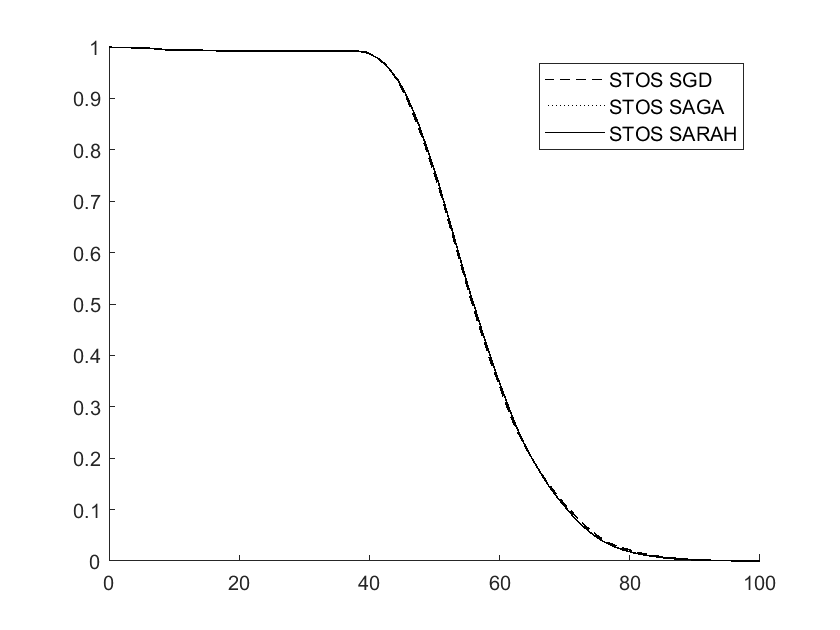}}
%  [b]
%  [b]
  \caption{Lung. Left, DVH (X-axis: dose; Y-axis: percentage); Right, DRVH (X-axis: dose rate; Y-axis: percentage). DVH and DRVH show the proposed methods with all the gradient methods improve dose rate for ROI and dose at the organ at risk (cord, lung, esophagus). }
  \label{fig:dvh_drvh_lung}
\end{figure}

\begin{figure}[htbp]
  \centering
  \begin{minipage}[b]{0.9\linewidth}
  \subfloat[ADMM]
  {\includegraphics[width=0.33\textwidth]{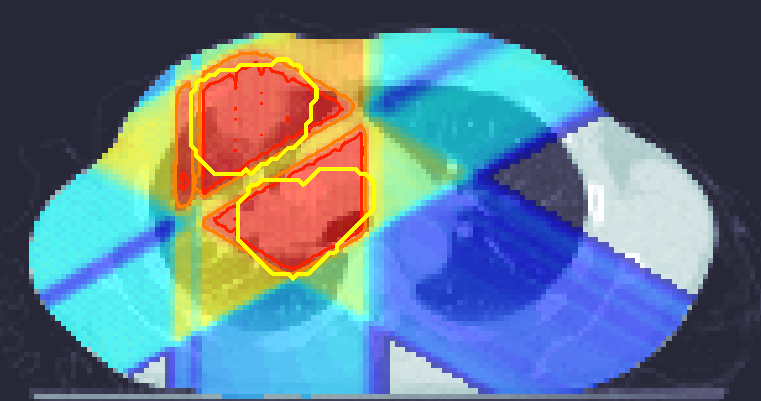}}
 % [b]
  \subfloat[STOS $n_{bs} = N$]
  {\includegraphics[width=0.33\textwidth]{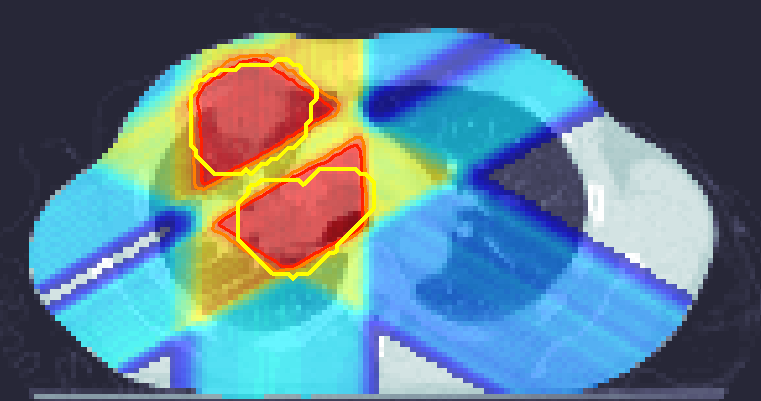}}
    \subfloat[STOS SGD $n_{bs} = \frac{N}{8}$]
  {\includegraphics[width=0.33\textwidth]{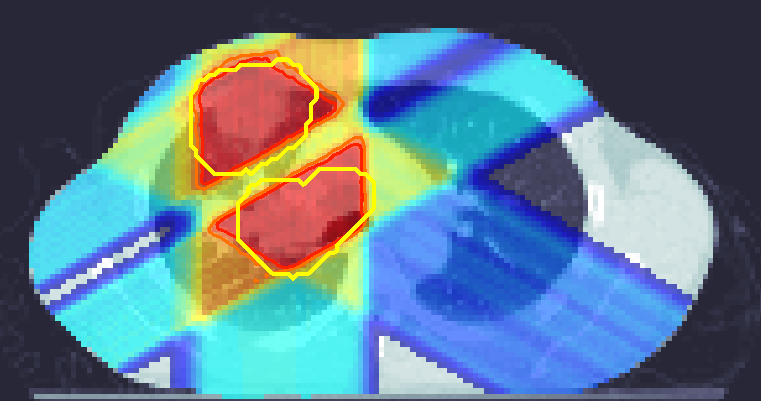}}\\
      \subfloat[STOS SGD $n_{bs} = \frac{N}{16}$]
  {\includegraphics[width=0.33\textwidth]{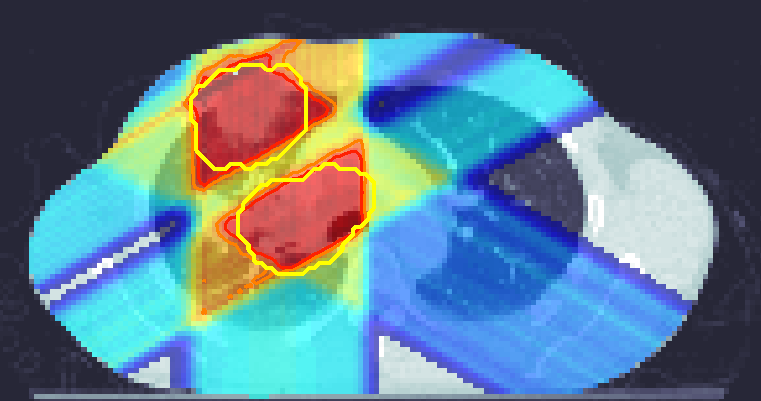}}
      \subfloat[STOS SAGA $n_{bs} = \frac{N}{8}$]
  {\includegraphics[width=0.33\textwidth]{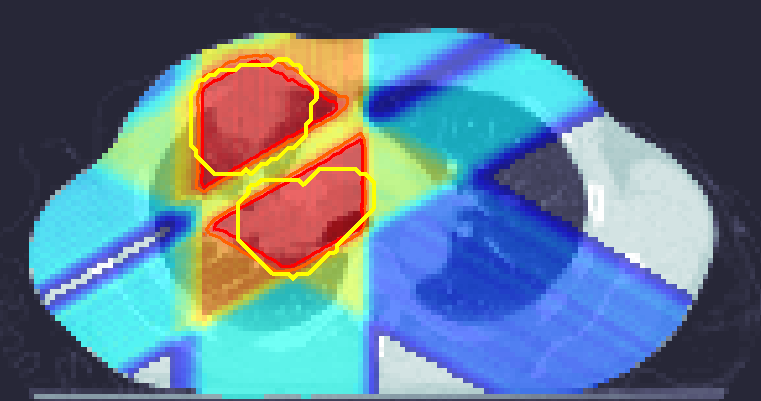}}
      \subfloat[STOS SARAH $n_{bs} = \frac{N}{8}$]
  {\includegraphics[width=0.33\textwidth]{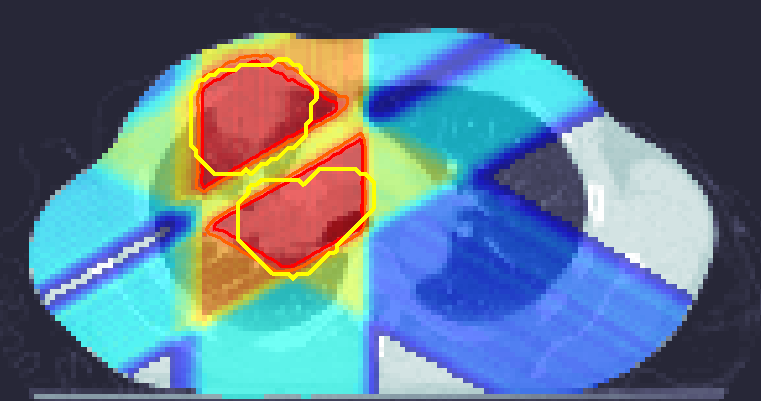}}
%  [b
  \end{minipage}
   \begin{minipage}[b]{0.08\linewidth}
   \includegraphics[width=0.4\textwidth]{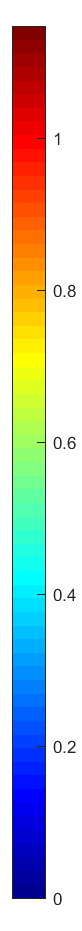}
   \end{minipage}
     \caption{FLASH-RT for lung. The dose distribution is plotted here using the window [0\%, 115\%]. 100\% isodose line, 90\% isodose line, and CTV are highlighted in the plots.}
       \label{fig:RT_lung}
\end{figure}

Here, we tested three different stochastic gradient estimators: SGD, SAGA, and SARAH. In Figure \ref{fig:energy_lung}, we first present the function values of all algorithms running the same number of epochs. From Figure \ref{fig:energy_lung} (a), it can be observed that both the STOS and TOS exhibit faster energy decay compared to ADMM. Additionally, the STOS achieves faster energy function descent compared to the TOS algorithm. In Figure \ref{fig:energy_lung} (b), we also show the results of STOS combined with SGD, SAGA, and SARAH. We can observe that the results of the three stochastic gradient estimators are similar, with SGD performing slightly better compared to the other two estimators. In summary, as shown in Figure \ref{fig:energy_lung}, the energy results indicate that the stochastic methods are superior to the non-stochastic methods for FLASH proton RT. This observation highlights the effectiveness of combining stochastic and probabilistic mechanisms in the optimization process. Moreover, decomposing large-scale problems into smaller sub-problems makes each sub-problem easier to solve, thereby further expanding the applicability of the algorithm.

To further demonstrate the effectiveness of STOS, graphical representations in the form of DVH and DVRH have been illustrated in Figure \ref{fig:dvh_drvh_lung} and values for some important indices are presented in Table \ref{tab:dose_lung}. All the results demonstrate that the proposed STOS can spare more high-dose OAR regions (e.g., ROI = PTV10mm) near treatment targets than ADMM. The proposed methods achieved better dose conformality to the treatment target (e.g., CI). In Figure \ref{fig:dvh_drvh_lung} (a) and (b), we present a comparison between STOS combined with SGD and ADMM. In Figure \ref{fig:dvh_drvh_lung} (c) and (d), we provide a comparison between STOS combined with SGD, SAGA, and SARAH. In Figure \ref{fig:dvh_drvh_lung} (a), it can be observed that the treatment plans obtained by STOS combined with SGD, SAGA, and SARAH exhibit lower dose on all OARs compared to those obtained by the ADMM algorithm. Although the results in Figure \ref{fig:dvh_drvh_lung} (c) suggest that there is no significant difference between STOS combined with SGD and combined with SAGA or SARAH, the numerical results provided in  Table \ref{tab:dose_lung} indicate that STOS combined with SGD performs significantly better than SAGA and SARAH. For example, the mean FLASH of ROI by SGD reduced from 1.29 Gy to 1.19 Gy for the lung when compared to ADMM.  For $D_{max}$, STOS combined with SGD outperforms all other algorithms, such as $D_{max}^{cord}$ is significantly lower compared to other algorithms. SAGA and SARAH obtain similar results, but they achieve lower $D_{max}^{lung}$ compared to SGD and other algorithms. We also present the dose distribution in Figure \ref{fig:RT_lung}. From Figure \ref{fig:RT_lung}, we can also observe that STOS combined with SGD provides a closer fit to the target region compared to SAGA or SARAH.

\subsubsection{FLASH proton radiotherapy for the brain} 
Brain tumors are also a challenging problem in the medical field. Compared to tumors in the lungs, brain tumors present additional difficulties due to the intricate network of structures and pathways within the brain, making treatment more challenging. In this experiment, we conducted FLASH proton radiation therapy for the brain. The beam angles are empirically chosen to be  $\left(45^{\circ}, 135^{\circ}, 225^{\circ}, 315^{\circ}\right)$.  The regions of CTV, body, and the organ at risk (carotid, cranial nerve, oral cavity) are chosen for dose planning leading to the size of the dose influence matrix $A$ being $1323239\times 1505$, and the ROI is chosen for Flash effect optimization leading to the size of $B$ being $3089 \times 353 $.  The brain data with a length of $547397$ is shuffled and divided into 8 parts and 16 parts in the same way as the lung data. The $\mu_{dr} $ is set to $40$ Gy and the $\mu_d $ is set to $8$ Gy.  The parameter $\gamma $ for STOS is set to $1.0\times 10^3$.

\begin{figure}[htbp]
  \centering
  \subfloat[Varying batch sizes ]
  {\includegraphics[width=0.4\textwidth]{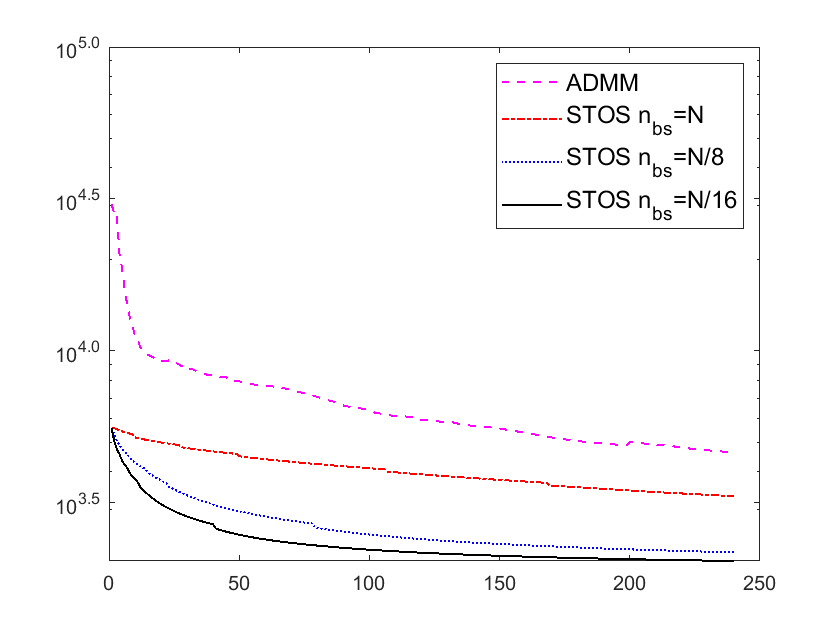}}
 % [b]     % 重点就在这，优先横向排列，自动换行
  \subfloat[ Varying gradient methods]
  {\includegraphics[width=0.4\textwidth]{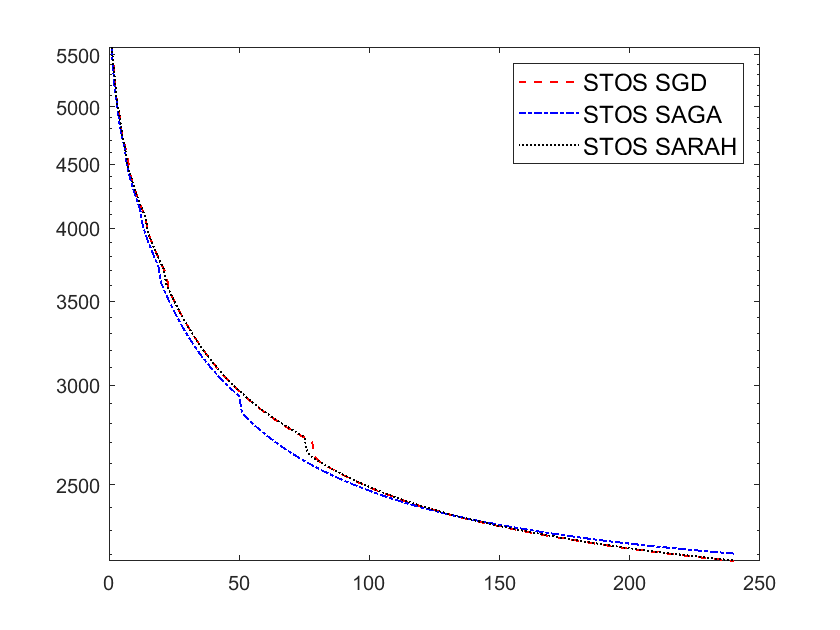}}
%  [b]
%  [b]
  \caption{ Test loss for different methods with the same epochs for the brain.}
  \label{fig:energy_brain}
\end{figure}

\begin{table}[htbp]
\centering
\resizebox{\textwidth}{!}{
\begin{tabular}{l|ccccccccc}
  \hline
    & $D_{max}$ & CI & $D_{mean}^{roi}$ & $D_{max}^{carotid}$   & $D_{max}^{craner}$ & $D_{mean}^{oral cav}$ & $D_{mean}^{body}$  & $P_{d}$ & $P_{dr}$   \\
  \hline
ADMM  &106.02 & 0.68 & 12.11 & 59.75 & 59.07  & 11.02 & 0.49 & 65.88 & 96.21  \\
STOS SGD $N$ &108.77 & 0.76 & 11.29 & 60.13 & 56.41  & 10.19 & 0.45 & 70.21 & 96.83 \\
STOS SGD $\frac{N}{8}$ &107.89 & 0.79 & 11.34 & 60.00 & 57.29  & 10.17 & 0.45 & 72.21 & 97.17  \\
STOS SGD $\frac{N}{16}$ &114.85 & 0.74 & 11.67 & 60.02 & 59.61  & 10.61 & 0.46 &72.32 & 97.66  \\
STOS SAGA $\frac{N}{8}$ &108.76 & 0.76 & 11.29 & 60.13 & 56.40  & 10.18 & 0.45 & 73.41 & 97.55  \\
STOS SARAH $\frac{N}{8}$  &110.63 & 0.76 & 12.49 & 58.35 & 53.77  & 11.94 & 0.49 & 72.93 & 97.83  \\
  \hline
\end{tabular}
}
\caption{\label{tab:dose_brain} Physical dose parameters for Brain (dose unit: Gy )}
\end{table}

The objective function values are shown in Figure \ref{fig:energy_brain} when running the same number of epochs. Here, we compare STOS combined with SGD, SAGA, SARAH, and ADMM. All the stochastic algorithms demonstrate a much faster decrease in the objective function value compared to ADMM. In Figure \ref{fig:energy_brain} (b), the comparison results between SGD, SAGA, and SARAH are displayed, indicating that for brain tumors, SAGA leads to a faster decrease in the energy function compared to SGD and SARAH. Further, we provide a graph of DVH and DVRH in Figure \ref{fig:dvh_drvh_brain}. Some important indices also be presented in Table \ref{tab:dose_brain}. All the results demonstrate that STOS combined with SGD, SAGA, and SARAH improve the dose rate to the ROI and reduce the dose to critical organs (carotid, cranial nerves, oral cavity). For instance, from Figure \ref{fig:dvh_drvh_brain}(b), it is evident that for the dose rate DVH (DRVH), STOS combined with SGD is noticeably higher than ADMM. In Figure \ref{fig:dvh_drvh_brain} (d), the results of STOS combined with SGD, SAGA, and SARAH are very similar. The numerical results in Table \ref{tab:dose_brain} also demonstrate these phenomena. For example, the $D_{max}$ (maximum dose) for SGD reaches 114.85 Gy, while ADMM achieves only 102.02 Gy. SGD significantly outperforms ADMM in this aspect. SARAH, on the other hand, reduces $D_{max}^{cranialner}$ to 53.77 Gy, which is substantially lower than that of ADMM, 59.07 Gy. Moreover, SARAH also yields much lower $D_{max}^{carotid}$ compared to ADMM. Furthermore, in Figure \ref{fig:RT_brain}, we present the results of dose distribution, which demonstrate that STOS combined with the SARAH gradient estimator achieves the best fit to the tumor region.

\begin{figure}[htbp]
  \centering
  \subfloat[Varying batch sizes ]
  {\includegraphics[width=0.4\textwidth]{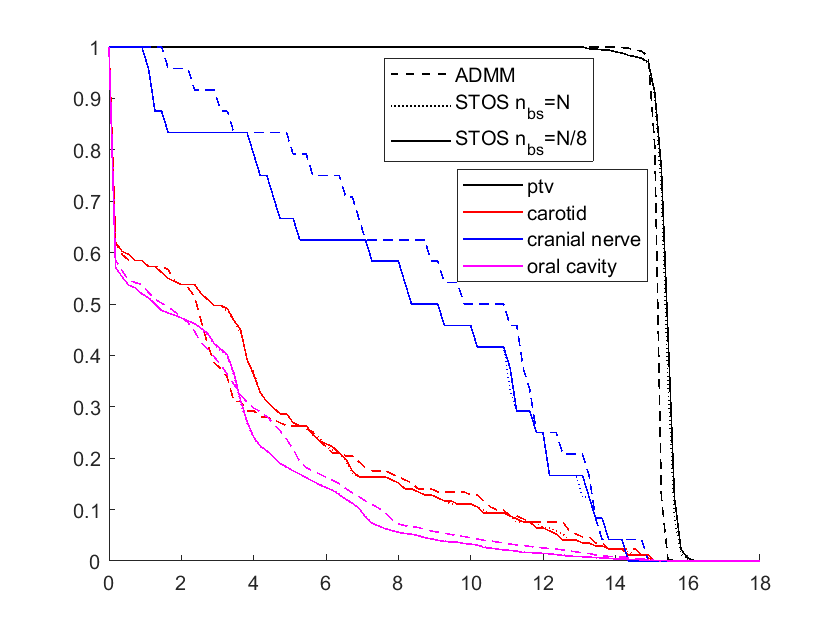}}
 % [b]
  \subfloat[Varying batch sizes ]
  {\includegraphics[width=0.4\textwidth]{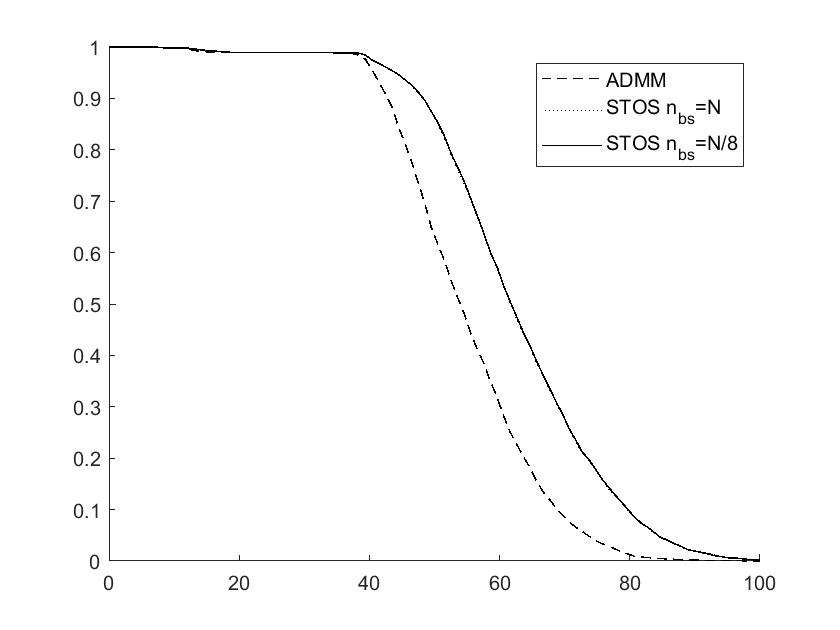}}\\
    \subfloat[Varying  gradient methods]
  {\includegraphics[width=0.4\textwidth]{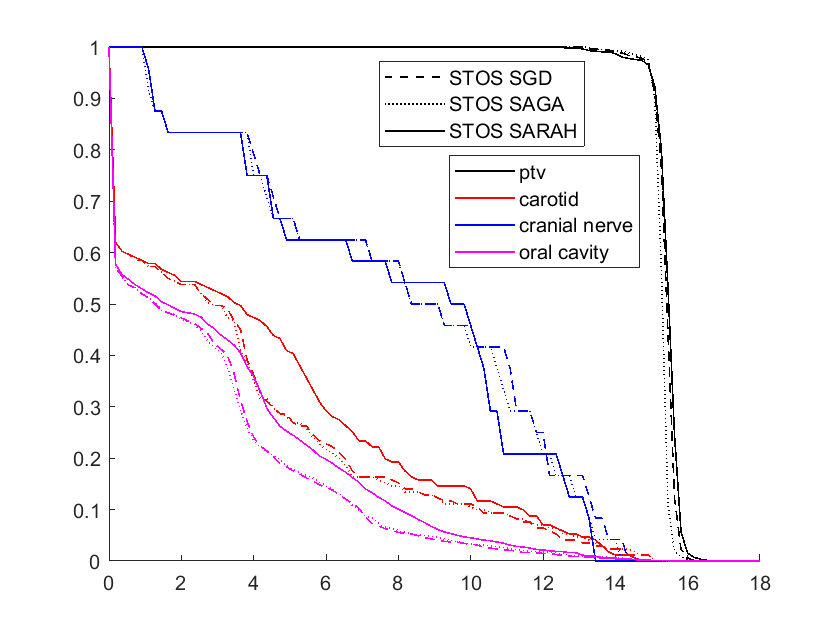}}
 % [b]
  \subfloat[Varying gradient methods]
  {\includegraphics[width=0.4\textwidth]{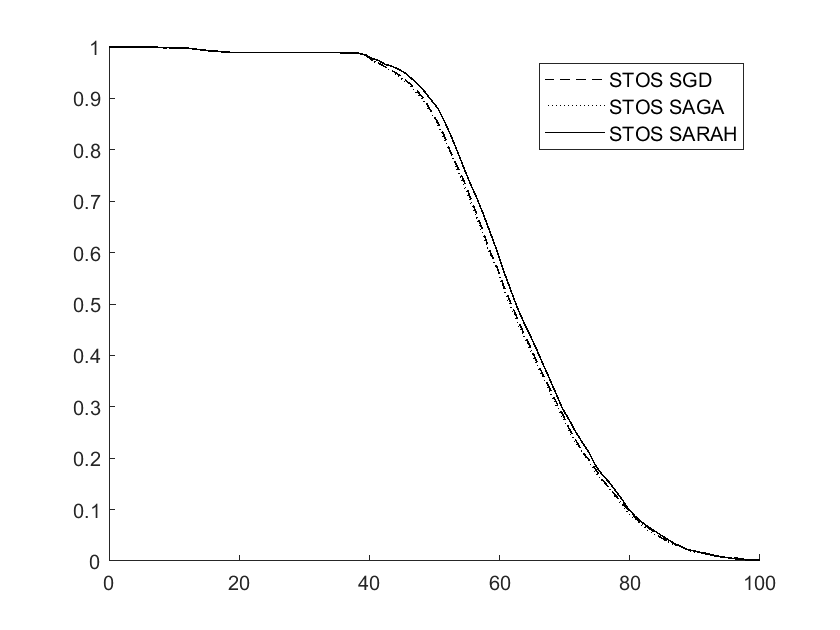}}
%  [b]
%  [b]
  \caption{Brain. Left, DVH (X-axis: dose; Y-axis: percentage); Right, DRVH (X-axis: dose rate; Y-axis: percentage). The DVH and DRVH show the proposed methods with all the gradient methods improve the dose rate for ROI and dose at the organ at risk (carotid, cranial nerve, oral cavity). }
   \label{fig:dvh_drvh_brain}
\end{figure}

\begin{figure}[htbp]
  \centering
  \begin{minipage}[b]{0.7\linewidth}
  \subfloat[ADMM]
  {\includegraphics[width=0.28\textwidth]{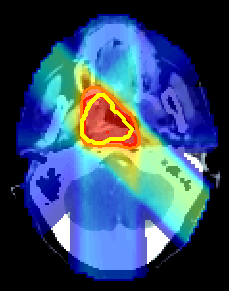}}
 % [b]
  \subfloat[STOS $n_{bs} = N$]
  {\includegraphics[width=0.28\textwidth]{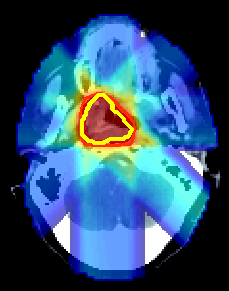}}
    \subfloat[STOS SGD $n_{bs} = \frac{N}{8}$]
  {\includegraphics[width=0.28\textwidth]{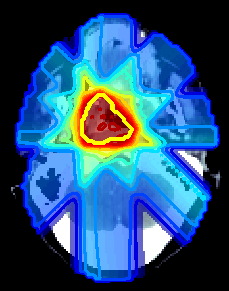}}\\
      \subfloat[STOS SGD $n_{bs} = \frac{N}{16}$]
  {\includegraphics[width=0.28\textwidth]{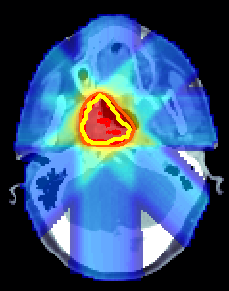}}
      \subfloat[STOS SAGA $n_{bs} = \frac{N}{8}$]
  {\includegraphics[width=0.28\textwidth]{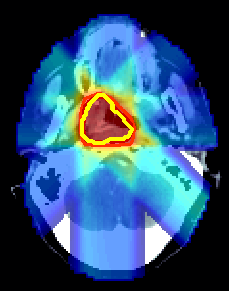}}
      \subfloat[STOS SARAH $n_{bs} = \frac{N}{8}$]
  {\includegraphics[width=0.28\textwidth]{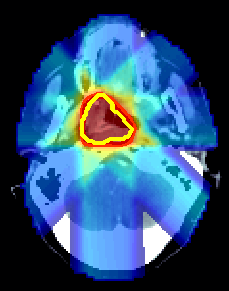}}
  \end{minipage}
   \begin{minipage}[b]{0.075\linewidth}
   \includegraphics[width=0.4\textwidth]{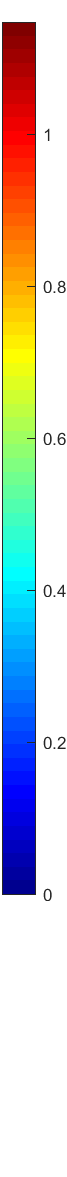}
   \end{minipage}
     \caption{FLASH-RT for the brain. The dose distribution is plotted here using the window [0\%, 115\%]. 100\% isodose line, 90\% isodose line, and CTV are highlighted in the plots.}
     \label{fig:RT_brain}
\end{figure}

\section{Conclusion}
\label{sec:conclude}
In this paper, we propose a stochastic three-operator splitting (STOS) algorithm, which combines two types of stochastic gradient estimators: unbiased gradient estimators and variance-reduced gradient estimators. We establish the convergence analysis and obtain the convergence rates for each type of gradient estimator. Furthermore, we validate the effectiveness of the STOS algorithm in the FLASH proton radiotherapy optimization. The extensive experiments demonstrate that STOS has attained state-of-the-art performance for the simultaneous dose and dose rate optimization for FLASH proton radiotherapy. The STOS algorithm has the capability to partition a large-scale problem into smaller subproblems, and the utilization of stochastic gradients leads to improved convergence.

\section*{Acknowledgments}
The work of Jian-Feng Cai was supported in part by the Hong Kong Research Grants Council GRF Grants 16310620 and 16306821, in part by the Hong Kong ITF MHP/009/20, and in part by the Project of Hetao Shenzhen-Hong Kong Science and Technology Innovation Cooperation Zone under Grant HZQB-KCZYB-2020083. Xiaoqun Zhang was supported by Shanghai Municipal Science and Technology Major Project (2021SHZDZX0102) and  NSFC (No. 12090024). Jiulong Liu was partially supported by the  Key grant of Chinese Ministry of Education (2022YFC2504302) and the Fund of the Youth Innovation Promotion Association, CAS (2022002). Fengmiao Bian was also partially supported by an outstanding Ph.D. graduate development scholarship from Shanghai Jiao Tong University.

\begin{small}
\bibliographystyle{plain}
\bibliography{reference}
\end{small}

\appendix
\section{Proofs of main theorems}
\label{app1}

\subsection{Basic Lemmas}
\label{app1.1}

We introduce some basic lemmas that are required in the convergence analysis. The proofs of Lemma \ref{lem1} and Lemma \ref{est-yz} can be found in reference \cite{BZ} and are not detailed here.

\begin{lemma}\label{lem1}
Let $a, b, c, d \in \mathbb{R}^n$. Then we have
\begin{equation*}
\begin{aligned}
&\| 2a-b-c-d\|^2 - \| a-c-d\|^2\\
&=\| a-c\|^2 -\|b-c\|^2 + 2\| a-b\|^2 + 2\langle d, b-a \rangle.
\end{aligned}
\end{equation*}
\end{lemma}

\begin{lemma}\label{est-yz}
Suppose that the function $F$ satisfies Assumption \ref{ass-fun} $(a1)$. Then the sequence $\{ (x_t,$ $ y_t, z_t)\}_{t \geq 1}$ generated by the Algorithm \ref{alg:STOS} satisfies
\begin{equation*}
\| x_t - x_{t-1} \| \leq (1 + \gamma L ) \| y_{t+1} - y_t \|
\end{equation*}
and 
\begin{equation*}
\| y_{t+1} - z_t \|^2 \leq \big[ (-1 + 2\gamma l ) + (1 + \gamma L)^2 \big] \| y_{t+1} - y_t \|^2.
\end{equation*}
\end{lemma}

Next, we present a crucial lemma in the convergence analysis of the STOS algorithm with both the unbiased estimator and the variance-reduced estimator.

\begin{lemma}\label{lem3}
Suppose that the  functions $F$, $G$, and $H$ satisfy Assumption \ref{ass-fun}. Let $\{ (x_t, y_t, z_t)\}_{t \geq 1}$ be a sequence generated by Algorithm \ref{alg:STOS}. Then for all $t \geq 1$, we have
\begin{equation}\label{eq11}
\begin{aligned}
&F(y_{t+1}) + G(z_{t+1}) + H(y_{t+1}) + \frac{1}{2\gamma} \| y_{t+1} - x_{t+1} \|^2 - \frac{1}{2\gamma} \| z_{t+1} - x_{t+1} \|^2 + \langle \wtilde H(y_{t+1}), z_{t+1} - y_{t+1} \rangle\\
&\leq F(y_t) + G(z_t) + H(y_t) + \frac{1}{2\gamma} \| y_t - x_t \|^2 - \frac{1}{2\gamma}\|z_t - x_t \|^2 + \langle \wtilde H(y_t), z_t - y_t \rangle +  \frac{3\gamma+2}{2\gamma} \| y_{t+1} - z_t \|^2\\
&\quad  +\frac{1}{2} \| \wtilde H(y_{t+1}) -\nabla H(y_{t+1}) \|^2 + \| \wtilde H(y_{t}) - \nabla H(y_{t}) \|^2 + \frac{\gamma(1+ l + 2\beta)-1}{2\gamma} \| y_{t+1} - y_t \|^2.
\end{aligned}
\end{equation}
\end{lemma}
\begin{proof}
From the proof of \cite[Lemma 3.3]{BZ}, we have
\begin{equation}\label{eq2}
\begin{aligned}
&F(y_{t+1}) + G(z_{t+1}) + \frac{1}{2\gamma} \| 2y_{t+1} - z_{t+1} - x_{t+1} - \gamma \wtilde H(y_{t+1}) \|^2 - \frac{1}{\gamma}\|y_{t+1} - z_{t+1}\|^2\\
&\quad - \frac{1}{2\gamma} \| x_{t+1} - y_{t+1} + \gamma \wtilde H(y_{t+1}) \|^2 \\
&\leq  F(y_t) + G(z_t) + \frac{1}{2\gamma} \| 2y_t - z_t -x_t - \gamma \wtilde H(y_t) \|^2 - \frac{1}{2\gamma} \| x_t - y_t + \gamma \wtilde H(y_t) \|^2 - \frac{1}{\gamma} \| y_t - z_t \|^2\\
&\quad + \langle \wtilde H(y_{t+1}), z_{t} - y_{t+1} \rangle -\langle \wtilde H(y_t), z_t - y_t \rangle + \frac{1}{\gamma} \| y_{t+1} - z_t \|^2 - \frac{1}{2} \left(\frac{1}{\gamma} -l \right)\|y_{t+1} - y_t \|^2.
\end{aligned}
\end{equation}
While the gradient $\wtilde H$ here is stochastic, the proof of the above inequality in \cite{BZ} still holds. Next, we further estimate $\langle \wtilde H(y_{t+1}), z_t - y_{t+1} \rangle - \langle \wtilde H(y_t), z_t - y_t \rangle$. Notice that
\begin{equation}\label{eq12}
\begin{aligned}
&\langle \wtilde H(y_{t+1}), z_t - y_{t+1} \rangle - \langle \wtilde H(y_t), z_t - y_t \rangle\\
&= \langle \wtilde H(y_{t+1}) - \nabla H(y_{t+1}), z_t - y_{t+1} \rangle - \langle \wtilde H(y_t)  - \nabla H(y_t), z_t - y_t \rangle + \langle \nabla H(y_t), y_t - y_{t+1} \rangle\\
&\quad + \langle \nabla H(y_{t+1}) - \nabla H(y_t) , z_t - y_{t+1} \rangle\\
&\leq \langle \wtilde H(y_{t+1}) - \nabla H(y_{t+1}), z_t - y_{t+1} \rangle - \langle \wtilde H(y_t)  - \nabla H(y_t), z_t - y_t \rangle\\
& \quad + H(y_t) - H(y_{t+1}) + \frac{\beta}{2} \| y_{t+1} - y_t \|^2 + \langle \nabla H(y_{t+1}) - \nabla H(y_t), z_t - y_{t+1} \rangle\\
&\leq  \langle \wtilde H(y_{t+1}) - \nabla H(y_{t+1}), z_t - y_{t+1} \rangle - \langle \wtilde H(y_t)  - \nabla H(y_t), z_t - y_t \rangle \\
&\quad + H(y_t) - H(y_{t+1}) + \beta\| y_{t+1} - y_t \|^2  + \frac{1}{2}\| y_{t+1} - z_t \|^2
\end{aligned}
\end{equation}
where we have used $(a3)$ in Assumption \ref{ass-fun}. Furthermore, using some basic inequalities, we get
\begin{equation}\label{eq15}
\begin{aligned}
&\langle \wtilde H(y_{t+1}) - \nabla H(y_{t+1}), z_t - y_{t+1} \rangle - \langle \wtilde H(y_t) - \nabla H(y_t), z_t - y_t \rangle\\
&= \langle \wtilde H(y_{t+1}) - \nabla H(y_{t+1}), z_t - y_{t+1} \rangle - \langle \wtilde H(y_t) - \nabla H(y_t), z_t - y_{t+1} \rangle \\
&\quad - \langle \wtilde H(y_t) - \nabla H(y_t), y_{t+1} - y_t \rangle\\
&\leq \frac{1}{2}\| \wtilde H(y_{t+1}) - \nabla H(y_{t+1}) \|^2 + \frac{1}{2} \| y_{t+1} - z_t \|^2 + \frac{1}{2} \| \wtilde H(y_{t}) - \nabla H(y_{t}) \|^2 + \frac{1}{2} \| y_{t+1} - z_t \|^2 \\
&\quad + \frac{1}{2} \| \wtilde H(y_{t}) - \nabla H(y_{t}) \|^2 + \frac{1}{2} \| y_{t+1} - y_t \|^2\\
&\leq \frac{1}{2}\| \wtilde H(y_{t+1}) - \nabla H(y_{t+1}) \|^2 + \| \wtilde H(y_{t}) - \nabla H(y_{t}) \|^2 + \| y_{t+1} - z_t \|^2 + \frac{1}{2} \| y_{t+1} - y_t \|^2.
\end{aligned}
\end{equation}
Substituting \eqref{eq15} into \eqref{eq12}, we get
\begin{equation}\label{eq16}
\begin{aligned}
&\langle \wtilde H(y_{t+1}), z_t - y_{t+1} \rangle - \langle \wtilde H(y_t), z_t - y_t \rangle \\
&\leq \frac{1}{2}\| \wtilde H(y_{t+1}) - \nabla H(y_{t+1}) \|^2 + \| \wtilde H(y_{t}) - \nabla H(y_{t}) \|^2   +  H(y_t) - H(y_{t+1}) \\
&\quad + \frac{2\beta + 1}{2} \| y_{t+1} - y_t \|^2 + \frac{3}{2}\| y_{t+1} - z_t \|^2.
\end{aligned}
\end{equation}
Combining \eqref{eq16} with \eqref{eq2},  we have
\begin{equation*}
\begin{aligned}
&F(y_{t+1}) + G(z_{t+1}) + H(y_{t+1}) + \frac{1}{2\gamma}\| 2y_{t+1} - z_{t+1} - x_{t+1} - \gamma \wtilde H(y_{t+1}) \|^2 \\
&\quad - \frac{1}{2\gamma}\| y_{t+1}  - x_{t+1} - \gamma \wtilde H(y_{t+1}) \|^2  - \frac{1}{\gamma}\| y_{t+1} - z_{t+1} \|^2\\
&\leq F(y_{t}) + G(z_{t}) + +  H(y_t) + \frac{1}{2\gamma}\| 2y_{t} - z_{t} - x_{t} - \gamma \wtilde H(y_{t}) \|^2 - \frac{1}{2\gamma}\| y_{t}  - x_{t} - \gamma \wtilde H(y_{t}) \|^2\\
&\quad - \frac{1}{\gamma} \| y_{t} - z_t \|^2 + \frac{1}{2}\| \wtilde H(y_{t+1}) - \nabla H(y_{t+1}) \|^2 + \| \wtilde H(y_{t}) - \nabla H(y_{t}) \|^2 + \frac{3\gamma+2}{2\gamma}  \| y_{t+1} - z_t \|^2\\
&\quad + \frac{\gamma(1+ l + 2\beta)-1}{2\gamma} \| y_{t+1} - y_t \|^2.
\end{aligned}
\end{equation*}
Using Lemma \ref{lem1}, we obtain
\begin{equation*}
\begin{aligned}
&F(y_{t+1}) + G(z_{t+1}) + H(y_{t+1}) + \frac{1}{2\gamma} \|y_{t+1} - x_{t+1} \|^2 - \frac{1}{2\gamma} \|z_{t+1} - x_{t+1} \|^2+ \langle \wtilde H(y_{t+1}), z_{t+1} - y_{t+1} \rangle\\
&\leq F(y_{t}) + G(z_{t}) +  H(y_t) + \frac{1}{2\gamma} \|y_t - x_t \|^2 - \frac{1}{2\gamma} \|z_t - x_t \|^2  + \langle \wtilde H(y_t), z_t - y_t \rangle +  \frac{3\gamma+2}{2\gamma} \| y_{t+1} - z_t \|^2 \\
&\quad + \frac{1}{2}\| \wtilde H(y_{t+1}) - \nabla H(y_{t+1}) \|^2 + \| \wtilde H(y_{t}) - \nabla H(y_{t}) \|^2 + \frac{\gamma(1+ l + 2\beta)-1}{2\gamma} \| y_{t+1} - y_t \|^2.
\end{aligned}
\end{equation*}
This completes the proof.
\end{proof}

\subsection{Convergence rates for unbiased gradient estimator}
\label{app1.2}
\begin{proof}[Proof of Theorem \ref{conv-unbiased}]
From Lemma \ref{lem3}, we have
\begin{equation*}
\begin{aligned}
&F(y_{t+1}) + G(z_{t+1}) + H(y_{t+1}) + \frac{1}{2\gamma} \| y_{t+1} - x_{t+1} \|^2 - \frac{1}{2\gamma} \| z_{t+1} - x_{t+1} \|^2 + \langle \wtilde H(y_{t+1}), z_{t+1} - y_{t+1} \rangle\\
&\leq F(y_t) + G(z_t) + H(y_t) + \frac{1}{2\gamma} \| y_t - x_t \|^2 - \frac{1}{2\gamma}\|z_t - x_t \|^2 + \langle \wtilde H(y_t), z_t - y_t \rangle + \frac{3\gamma +2}{2\gamma}  \| y_{t+1} - z_t \|^2\\
&\quad +\frac{1}{2} \| \wtilde H(y_{t+1}) - \nabla H(y_{t+1}) \|^2 + \| \wtilde H(y_{t}) - \nabla H(y_{t}) \|^2 + \frac{\gamma (1+ l + 2\beta)- 1}{2\gamma} \| y_{t+1} - y_t \|^2.
\end{aligned}
\end{equation*}
Since 
\begin{equation*}
\langle \wtilde H(y_t), z_t - y_t \rangle = \langle \wtilde H(y_t) - \nabla H(y_t), z_t - y_t \rangle + \langle \nabla H(y_t), z_t - y_t \rangle,
\end{equation*}
we have 
\begin{equation}\label{eq70}
\begin{aligned}
&F(y_{t+1}) + G(z_{t+1}) + H(y_{t+1}) + \frac{1}{2\gamma} \| y_{t+1} - x_{t+1} \|^2 - \frac{1}{2\gamma} \| z_{t+1} - x_{t+1} \|^2 + \langle \nabla H(y_{t+1}), z_{t+1} - y_{t+1} \rangle\\
&\quad  + \langle \wtilde H(y_{t+1}) - \nabla H(y_{t+1}), z_{t+1} - y_{t+1}\rangle\\
&\leq F(y_t) + G(z_t) + H(y_t) + \frac{1}{2\gamma}\| y_t -  x_t \|^2 - \frac{1}{2\gamma} \| z_t - x_t \|^2 + \langle \nabla H(y_t), z_t  - y_t \rangle\\
&\quad + \langle \wtilde H(y_t) - \nabla H(y_t), z_t - y_t \rangle + \frac{1}{2}\| \wtilde H(y_{t+1}) - \nabla H(y_{t+1})\|^2 + \| \wtilde H(y_t) - \nabla H(y_t) \|^2\\
&\quad + \frac{3\gamma +2}{2\gamma}  \| y_{t+1} - z_t \|^2 + \frac{\gamma (1+ l + 2\beta)- 1}{2\gamma} \| y_{t+1} - y_t \|^2.
\end{aligned}
\end{equation}
Denote 
\begin{equation*}
\Phi^{inf} = \inf_{t \geq 1} \{ F(y_t) + G(z_t) + H(y_t) + \frac{1}{2\gamma} \| y_t - x_t \|^2 -  \frac{1}{2\gamma} \| z_t - x_t \|^2 + \langle \nabla H(y_t), z_t - y_t \rangle \}
\end{equation*}
and 
\begin{equation*}
\delta_t = F(y_t) + G(z_t) + H(y_t) + \frac{1}{2\gamma}\| y_t - x_t \|^2 - \frac{1}{2\gamma} \| z_t - x_t \|^2 + \langle \nabla H(y_t), z_t - y_t \rangle - \Phi^{inf}.
\end{equation*}
Then \eqref{eq70} can be rewritten as
\begin{equation}\label{eq71}
\begin{aligned}
\delta_{t+1} &\leq \delta_t + \langle \wtilde H(y_t) - \nabla H(y_t), z_t - y_t \rangle  - \langle \wtilde H(y_{t+1}) - \nabla H(y_{t+1}), z_{t+1} -  y_{t+1} \rangle \\
&\quad + \frac{1}{2} \| \wtilde H(y_{t+1}) - \nabla H(y_{t+1}) \|^2 + \| \wtilde H(y_t) - \nabla H(y_t) \|^2 + \frac{3\gamma +2}{2\gamma}  \| y_{t+1} - z_t \|^2\\
&\quad + \frac{\gamma (1+ l + 2\beta)- 1}{2\gamma} \| y_{t+1} - y_t \|^2.
\end{aligned}
\end{equation}
Combining \eqref{eq71} with the estimate of $\| y_{t+1} - z_t \|^2$ in Lemma \ref{est-yz}, we have
\begin{equation*}
\begin{aligned}
\delta_{t+1} 
&\leq \delta_t + \langle \wtilde H(y_t) - \nabla H(y_t), z_t - y_t \rangle - \langle \wtilde H(y_{t+1}) - \nabla H(y_{t+1}), z_{t+1} - y_{t+1} \rangle\\
&\quad + \frac{1}{2} \| \wtilde H(y_{t+1}) - \nabla H(y_{t+1}) \|^2 + \| \wtilde H(y_t) - \nabla H(y_t) \|^2\\
&\quad + \Big( \frac{(3\gamma +2)(\gamma L^2 + 2l + 2L)}{2} +  \frac{\gamma (1+ l + 2\beta)- 1}{2\gamma} \Big) \|y_{t+1} - y_t \|^2.
\end{aligned}
\end{equation*}
Denote $\Lambda(\gamma)= -  \Big( \frac{(3\gamma +2)(\gamma L^2 + 2l + 2L)}{2} +  \frac{\gamma (1+ l + 2\beta)- 1}{2\gamma} \Big)$, we have
\begin{equation*}
\begin{aligned}
\delta_{t+1} + \Lambda(\gamma) \| y_{t+1} - y_t \|^2 &\leq \delta_t + \langle \wtilde H(y_t) - \nabla H(y_t), z_t - y_t \rangle - \langle \wtilde H(y_{t+1}) - \nabla H(y_{t+1}), z_{t+1} - y_{t+1} \rangle \\
&\quad + \frac{1}{2} \| \wtilde H(y_{t+1}) - \nabla H(y_{t+1}) \|^2 + \| \wtilde H(y_t) - \nabla H(y_t) \|^2.
\end{aligned}
\end{equation*}
Further, we have
\begin{equation}\label{eq74}
\begin{aligned}
\| y_{t+1} - y_t \|^2
&\leq \frac{1}{\Lambda(\gamma)} \big[ \delta_t - \delta_{t+1} + \langle \wtilde H(y_t) - \nabla H(y_t), z_t - y_t \rangle - \langle \wtilde H(y_{t+1})- \nabla H(y_{t+1}), z_{t+1} - y_{t+1} \rangle\\
&\quad  + \frac{1}{2} \| \wtilde H(y_{t+1}) - \nabla H(y_{t+1})\|^2 + \| \wtilde H(y_t) - \nabla H(y_t)\|^2 \big].
\end{aligned}
\end{equation}
By the optimality conditions \eqref{opti},  we know
\begin{equation*}
\frac{1}{\gamma} (y_{t+1} - z_{t+1}) \in \partial G(z_{t+1}) + \nabla F(y_{t+1}) + \wtilde H(y_{t+1}).
\end{equation*}
This implies that
\begin{equation*}
\begin{aligned}
&\frac{1}{\gamma} (y_{t+1} - z_{t+1}) + \nabla F(z_{t+1}) - \nabla F(y_{t+1}) + \nabla H(z_{t+1}) - \nabla H(y_{t+1}) \\
& + \nabla H(y_{t+1}) - \wtilde H(y_{t+1}) \in \partial G(z_{t+1}) + \nabla F(z_{t+1}) + \nabla H(z_{t+1}).
\end{aligned}
\end{equation*}
By Cauchy inequality, the update \eqref{algx} and Lemma \ref{est-yz}, we have
\begin{equation*}
\begin{aligned}
&\dist^2\big(0, \partial G(z_{t+1}) + \nabla F(z_{t+1}) + \nabla H(z_{t+1}) \big)\\
&\leq \frac{1}{2\gamma^2} \big\|  y_{t+1} - z_{t+1} \big\|^2 + \frac{1}{2} \big\| \nabla F(z_{t+1}) - \nabla F(y_{t+1}) \big\|^2 + \frac{1}{2} \big\| \nabla H(z_{t+1}) - \nabla H(y_{t+1}) \big\|^2\\
&\quad + \frac{1}{2} \big\| \nabla H(y_{t+1}) - \wtilde H(y_{t+1}) \big\|^2 \\
&\leq \frac{1 + \gamma^2(L^2 + \beta^2)}{2\gamma^2}  \big\| z_{t+1} - y_{t+1} \big\|^2 + \frac{1}{2} \big\| \nabla H(y_{t+1}) - \wtilde H(y_{t+1}) \big\|^2\\
&\leq  \frac{\big[1 + \gamma^2(L^2 + \beta^2)\big](1+ \gamma L)^2}{2\gamma^2}  \big\| y_{t+1} - y_{t} \big\|^2 + \frac{1}{2} \big\| \nabla H(y_{t+1}) - \wtilde H(y_{t+1}) \big\|^2.
\end{aligned}
\end{equation*}
Taking expectations on both sides (over $J_t$) and using \eqref{eq74}, we have
\begin{equation*}
\begin{aligned}
&\mathbb{E} \dist^2 \big(0, \partial G(z_{t+1})) + \nabla F(z_{t+1}) + \nabla H(z_{t+1}) \big)\\
&\leq  \frac{\big[1 + \gamma^2(L^2 + \beta^2)\big](1+ \gamma L)^2}{2\gamma^2}  \mathbb{E} \| y_{t+1} - y_{t} \|^2 + \frac{1}{2} \mathbb{E} \| \nabla H(y_{t+1}) - \wtilde H(y_{t+1}) \|^2 \\
&\leq  \frac{\big[ 1 + \gamma^2(L^2 + \beta^2) \big] (1+ \gamma L)^2 }{2\gamma^2 \Lambda(\gamma)} \big(\mathbb{E} \delta_t - \mathbb{E} \delta_{t+1} + \frac{1}{2} \mathbb{E}\| \wtilde H(y_{t+1}) - \nabla H(y_{t+1}) \|^2 \\
&\quad + \mathbb{E} \| \wtilde H(y_t) - \nabla H(y_t) \|^2 \big) + \frac{1}{2} \mathbb{E} \| \nabla H(y_{t+1}) - \wtilde H(y_{t+1}) \|^2\|.
\end{aligned}
\end{equation*}
From \cite[Lemma 3]{YMS} and Assumption \ref{ass-bv}, we know
\begin{equation*}
\mathbb{E} \| \wtilde H(y_t) - \nabla H(y_t) \| \leq \frac{\sigma^2}{b}.
\end{equation*}
Thus, we have 
\begin{equation*}
\begin{aligned}
&\mathbb{E} \dist^2 \big(0, \partial G(z_{t+1})) + \nabla F(z_{t+1}) + \nabla H(z_{t+1}) \big)\\
&\leq \frac{\big[ 1 + \gamma^2(L^2 + \beta^2) \big] (1+ \gamma L)^2 }{2\gamma^2 \Lambda(\gamma)} \big( \mathbb{E} \delta_t - \mathbb{E} \delta_{t+1} +  \frac{3}{2} \frac{\sigma^2}{b} \big)+ \frac{\sigma^2}{2 b}.
\end{aligned}
\end{equation*}
Summing from $t=0$ to $T$ and choosing $b =T$, we get
\begin{equation*}
\begin{aligned}
&\frac{1}{T} \sum_{t=0}^T  \mathbb{E}~\dist^2(0, \partial G(z_{t+1}) + \nabla F(z_{t+1}) + \nabla H(z_{t+1})) \\
&\leq  \frac{1 + \gamma^2(L^2 + \beta^2)}{2\gamma^2 \Lambda(\gamma)} \frac{(1+ \gamma L)^2}{T} \big(\delta_0 + \frac{3\sigma^2}{2}\big) +  \frac{\sigma^2}{2T}.
\end{aligned}
\end{equation*}
Take $\tau$ from $\{ 1, 2, \cdots, T\}$ randomly. Then we obtain
\begin{equation*}
\begin{aligned}
&\mathbb{E}_{\tau} \mathbb{E}~\dist^2(0, \partial G(z_{\tau}) + \nabla F(z_{\tau}) + \nabla H(z_{\tau})) \\
&\leq  \frac{1 + \gamma^2(L^2 + \beta^2)}{2\gamma^2 \Lambda(\gamma)} \frac{(1+ \gamma L)^2}{T} \big(\delta_0 + \frac{3\sigma^2}{2}\big) +  \frac{\sigma^2}{2T}\\
&=\mathcal{O}(T^{-1}).
\end{aligned}
\end{equation*}
The proof of Theorem \ref{conv-unbiased} is completed.
\end{proof}

\subsection{Convergence rates for variance-reduced gradient estimator}
\label{app1.3}
 We first recall a supermartingale convergence theorem which can be referred to \cite{D} and \cite{RS} for details.
\begin{lemma}[Supermartingale Convergence Theorem]\label{super}
Let $\mathbb{E}_t$ denote the expectation conditional on the first $t$ iterations of Algorithm \ref{alg:STOS}. Let $\{ U_t \}_{t=0}^{\infty}$ and $\{ V_t \}_{t=0}^{\infty}$ be sequences of bounded non-negative random variables such that $U_t$ and $V_t$ depend only on the first $t$ iterations of Algorithm \ref{alg:STOS}. If 
\begin{equation}\nonumber
\mathbb{E}_t U_{t+1} + V_t \leq U_t,
\end{equation}
then $\sum_{t=0}^{\infty} V_t< \infty~a.s.$ and $U_t$ converges $a.s.$.
\end{lemma}

\begin{proof}[Proof of Theorem \ref{decrease}]
We will complete the proof of this theorem using Lemma \ref{lem3}. According to Cauchy inequality, we get the following estimate
\begin{equation}\label{eq21}
\begin{aligned}
\langle \wtilde H(y_t), z_t - y_t \rangle &= \langle \wtilde H(y_t) - \nabla H(y_t), z_t -  y_t \rangle + \langle \nabla H(y_t), z_t - y_t \rangle\\
&\leq \frac{1}{2} \| \wtilde H(y_t) - \nabla H(y_t) \|^2 + \frac{1}{2} \| z_t - y_t \|^2 + \langle \nabla H(y_t), z_t - y_t \rangle.
\end{aligned}
\end{equation}
Combining \eqref{eq21} with \eqref{eq11} in Lemma \ref{lem3}, we have
\begin{equation}\label{eq22}
\begin{aligned}
&F(y_{t+1}) + G(z_{t+1}) + H(y_{t+1}) + \frac{1}{2\gamma} \|y_{t+1} - x_{t+1} \|^2 - \frac{1}{2\gamma} \|z_{t+1} - x_{t+1} \|^2+ \langle \nabla H(y_{t+1}), z_{t+1} - y_{t+1} \rangle\\
&\quad -  \frac{1}{2} \| \wtilde H(y_{t+1}) - \nabla H(y_{t+1}) \|^2\\
&\leq F(y_{t}) + G(z_{t}) +  H(y_t) + \frac{1}{2\gamma} \|y_t - x_t \|^2 - \frac{1}{2\gamma} \|z_t - x_t \|^2  + \langle \wtilde H(y_t), z_t - y_t \rangle\\
&\quad  -  \frac{1}{2}\| \wtilde H(y_{t}) - \nabla H(y_{t}) \|^2  + \langle \nabla H(y_{t+1}) - \wtilde H(y_{t+1}), z_{t+1} - y_{t+1} \rangle + \frac{3}{2}\| \wtilde H(y_{t}) - \nabla H(y_{t}) \|^2\\
&\quad  + \frac{3\gamma + 2}{2\gamma} \| y_{t+1} - z_t \|^2+  \frac{\gamma (1+ l + 2\beta) -1 }{2\gamma} \| y_{t+1} - y_t \|^2\\
&\leq F(y_{t}) + G(z_{t}) +  H(y_t) + \frac{1}{2\gamma} \|y_t - x_t \|^2 - \frac{1}{2\gamma} \|z_t - x_t \|^2 + \frac{1}{2} \|z_t - y_t \|^2 + \langle \nabla H(y_t), z_t - y_t \rangle \\
&\quad +  \frac{1}{2}\| \wtilde H(y_{t+1}) - \nabla H(y_{t+1}) \|^2 + \frac{1}{2}\| z_{t+1} - y_{t+1} \|^2 + \frac{3}{2}\| \wtilde H(y_{t}) - \nabla H(y_{t}) \|^2  + \frac{3\gamma + 2}{2\gamma} \| y_{t+1} - z_t \|^2 \\
&\quad +  \frac{\gamma (1+ l + 2\beta) -1 }{2\gamma} \| y_{t+1} - y_t \|^2.
\end{aligned}
\end{equation}
Simplifying the equation \eqref{eq22}, we get
\begin{equation*}
\begin{aligned}
&F(y_{t+1}) + G(z_{t+1}) + H(y_{t+1}) + \frac{1}{2\gamma} \|y_{t+1} - x_{t+1} \|^2 - \frac{1}{2\gamma} \|z_{t+1} - x_{t+1} \|^2+ \langle \nabla H(y_{t+1}), z_{t+1} - y_{t+1} \rangle\\
&\leq F(y_{t}) + G(z_{t}) +  H(y_t) + \frac{1}{2\gamma} \|y_t - x_t \|^2 - \frac{1}{2\gamma} \|z_t - x_t \|^2  + \frac{1}{2}\|z_t -  y_t \|^2 + \langle \nabla H(y_t), z_t - y_t \rangle \\
&\quad + \| \wtilde H(y_{t+1}) - \nabla H(y_{t+1})\|^2 + \frac{1}{2}\| z_{t+1} - y_{t+1} \|^2 + \frac{3}{2}\| \wtilde H(y_{t}) - \nabla H(y_{t}) \|^2 + \frac{3\gamma + 2}{2\gamma} \| y_{t+1} - z_t \|^2 \\
&\quad +  \frac{\gamma (1+ l + 2\beta) -1 }{2\gamma} \| y_{t+1} - y_t \|^2.
\end{aligned}
\end{equation*}
For any $C_1 > 0$, adding $C_1 \| y_{t+1} -  y_t \|^2$ to both sides of the equation yields
\begin{equation}\label{eq24}
\begin{aligned}
&F(y_{t+1}) + G(z_{t+1}) + H(y_{t+1}) + \frac{1}{2\gamma} \|y_{t+1} - x_{t+1} \|^2 - \frac{1}{2\gamma} \|z_{t+1} - x_{t+1} \|^2+ \langle \nabla H(y_{t+1}), z_{t+1} - y_{t+1} \rangle\\
&\quad - \frac{1}{2} \| z_{t+1} - y_{t+1} \|^2 + C_1 \| y_{t+1} - y_t \|^2\\
&\leq F(y_{t}) + G(z_{t}) +  H(y_t) + \frac{1}{2\gamma} \|y_t - x_t \|^2 - \frac{1}{2\gamma} \|z_t - x_t \|^2  +  \langle \nabla H(y_t), z_t - y_t \rangle + C_1 \| y_t - y_{t-1} \|^2\\
&\quad  -\frac{1}{2}\|z_t -  y_t \|^2 + \|z_t -  y_t \|^2 + \frac{3}{2}\| \wtilde H(y_{t}) - \nabla H(y_{t}) \|^2  + \frac{3\gamma + 2}{2\gamma} \| y_{t+1} - z_t \|^2\\
&\quad + \big(\frac{\gamma (1+ l + 2\beta) -1 }{2\gamma} + C_1 \big) \| y_{t+1} - y_t \|^2 +  \| \wtilde H(y_{t+1}) - \nabla H(y_{t+1})\|^2 - C_1\| y_t - y_{t-1} \|^2.
\end{aligned}
\end{equation}
Recalling the definition of $\Phi_t$ in \eqref{Phi} and taking the conditional expectation with respect to $x_t$ on both sides of \eqref{eq24} yields
\begin{equation*}
\begin{aligned}
\mathbb{E}_t \Phi_{t+1} &\leq \Phi_t + \| z_t - y_t \|^2 + \frac{3}{2} \mathbb{E}_t \| \wtilde H(y_t) - \nabla H(y_t) \|^2 + \frac{3\gamma + 2}{2\gamma}  \mathbb{E}_t \| y_{t+1} - z_t \|^2\\
& + \frac{\gamma (1+ l + 2\beta) -1 + 2\gamma C_1}{2\gamma}  \mathbb{E}_t \| y_{t+1} - y_t \|^2 +  \mathbb{E}_t \| \wtilde H(y_{t+1}) - \nabla H(y_{t+1})\|^2 - C_1\| y_t - y_{t-1} \|^2.
\end{aligned}
\end{equation*}
This is equivalent to
\begin{equation}\label{eq26}
\begin{aligned}
&\mathbb{E}_t [ \Phi_{t+1} -  \| \wtilde H(y_{t+1}) - \nabla H(y_{t+1})\|^2 ]\\
&\leq \Phi_t - \mathbb{E}_t \| \wtilde H(y_t) - \nabla H(y_t) \|^2 + \| z_t - y_t \|^2 + \frac{5}{2} \mathbb{E}_t \| \wtilde H(y_t) - \nabla H(y_t) \|^2 + \frac{3\gamma + 2}{2\gamma}  \mathbb{E}_t \| y_{t+1} - z_t \|^2\\
&\quad + \frac{\gamma (1+ l + 2\beta) -1 + 2\gamma C_1}{2\gamma}  \mathbb{E}_t \| y_{t+1} - y_t \|^2  - C_1\| y_t - y_{t-1} \|^2.
\end{aligned}
\end{equation}
According to \eqref{var-geo}, we have
\begin{equation}\label{eq27}
\Upsilon_t \leq \frac{1}{\rho} \Upsilon_t -   \frac{1}{\rho} \mathbb{E}_t \Upsilon_{t+1} + \frac{V_{\Upsilon}}{\rho} \mathbb{E}_t \|y_{t+1} - y_t \|^2 + \frac{V_{\Upsilon}}{\rho} \| y_t -  y_{t-1} \|^2.
\end{equation}
Combining \eqref{eq27} with \eqref{eq26}, we get
\begin{equation}\label{eq28}
\begin{aligned}
&\mathbb{E}_t [ \Phi_{t+1} -  \| \wtilde H(y_{t+1}) - \nabla H(y_{t+1})\|^2 ]\\
&\leq \Phi_t - \mathbb{E}_t \| \wtilde H(y_t) - \nabla H(y_t) \|^2  +  \frac{3\gamma + 2}{2\gamma}  \mathbb{E}_t \| y_{t+1} - z_t \|^2 + \frac{\gamma (1+ l + 2\beta) -1 + 2\gamma C_1}{2\gamma} \mathbb{E}_t \| y_{t+1} - y_t \|^2 \\
&\quad +\frac{5}{2} \Upsilon_t  + \frac{5V_{1}}{2} \mathbb{E}_t \|y_{t+1} - y_t \|^2 + \frac{5V_{1}}{2} \| y_t -  y_{t-1} \|^2 + \| z_t - y_t \|^2 - C_1 \| y_t - y_{t-1}\|^2\\
&\leq \Phi_t - \mathbb{E}_t \| \wtilde H(y_t) - \nabla H(y_t) \|^2  +  \frac{3\gamma + 2}{2\gamma}  \mathbb{E}_t \| y_{t+1} - z_t \|^2 + \frac{\gamma (1+ l + 2\beta) -1 + 2\gamma C_1}{2\gamma} \mathbb{E}_t \| y_{t+1} - y_t \|^2 \\
&\quad +\frac{5}{2\rho} \Upsilon_t  - \frac{5}{2\rho} \mathbb{E}_t \Upsilon_{t+1} + \frac{5V_{\Upsilon}}{2\rho} \mathbb{E}_t \|y_{t+1} - y_t \|^2 + \frac{5V_{\Upsilon}}{2\rho} \| y_t -  y_{t-1} \|^2 + \frac{5V_{1}}{2} \mathbb{E}_t \|y_{t+1} - y_t \|^2\\
&\quad  + \frac{5V_{1}}{2} \| y_t -  y_{t-1} \|^2    + \| z_t - y_t \|^2 - C_1 \| y_t - y_{t-1}\|^2.
\end{aligned}
\end{equation}
We rearrange  \eqref{eq28} as
\begin{equation*}
\begin{aligned}
&\mathbb{E}_t [ \Phi_{t+1} -  \| \wtilde H(y_{t+1}) - \nabla H(y_{t+1})\|^2 + \frac{5}{2\rho} \Upsilon_{t+1} + \frac{5V_{1}\rho + 5 V_{\Upsilon}}{2\rho}  \|y_{t+1} - y_t \|^2 ]\\
&\leq \Phi_t - \mathbb{E}_t \| \wtilde H(y_t) - \nabla H(y_t) \|^2  +\frac{5}{2\rho} \Upsilon_t + \frac{5V_{1}\rho + 5 V_{\Upsilon}}{2\rho} \|y_{t} - y_{t-1} \|^2  - C_1 \| y_t - y_{t-1}\|^2\\
&\quad +  \frac{3\gamma + 2}{2\gamma}  \mathbb{E}_t \| y_{t+1} - z_t \|^2 + \big(  \frac{\gamma (1+ l + 2\beta) -1 + 2\gamma C_1}{2\gamma} + \frac{5V_{1}\rho + 5 V_{\Upsilon}}{\rho}  \big) \mathbb{E}_t \|y_{t+1} - y_t \|^2  + \| z_t - y_t \|^2.
\end{aligned}
\end{equation*}
Recalling the definition of $\Psi_t$ in \eqref{Psi} and taking the expectation on both sides of \eqref{eq29} yields
\begin{equation}\label{eq30}
\begin{aligned}
\mathbb{E}\Psi_{t+1} &\leq \mathbb{E} \Psi_t + \big(  \frac{\gamma (1+ l + 2\beta) -1 + 2\gamma C_1}{2\gamma} + \frac{5V_{1}\rho + 5 V_{\Upsilon}}{\rho}  \big)  \mathbb{E} \|y_{t+1} - y_t \|^2 + \mathbb{E} \|z_{t} - y_t \|^2\\
&\quad +  \frac{3\gamma + 2}{2\gamma}  \mathbb{E} \| y_{t+1} - z_t \|^2 - C_1 \mathbb{E} \| y_t - y_{t-1}\|^2.
\end{aligned}
\end{equation}
According to Lemma \ref{est-yz}, inequality \eqref{eq30} can be rewritten as 
\begin{equation*}
\begin{aligned}
\mathbb{E}\Psi_{t+1} &\leq \mathbb{E}\Psi_t + \Big(\frac{\gamma (1+ l + 2\beta) -1 + 2\gamma C_1}{2\gamma} + \frac{5V_{1}\rho + 5 V_{\Upsilon}}{\rho} \\
&\quad + \frac{3\gamma + 2}{2\gamma}((-1 + 2\gamma l) + (1+ \gamma L)^2) + (1+\gamma L)^2\Big)\mathbb{E}\| y_{t+1} - y_t \|^2   - C_1 \mathbb{E} \| y_t - y_{t-1}\|^2.
\end{aligned}
\end{equation*}
Denote
\begin{equation*}
\Lambda_1(\gamma) :=\frac{1- \gamma (1+ l + 2\beta) - 2\gamma C_1}{2\gamma} - \frac{5V_{1}\rho + 5 V_{\Upsilon}}{\rho} -  \frac{3\gamma + 2}{2\gamma} ((-1 + 2\gamma l) + (1+ \gamma L)^2) - (1+\gamma L)^2.
\end{equation*}
Then we have
\begin{equation}\label{eq33}
\mathbb{E}\Psi_{t+1} \leq \mathbb{E}\Psi_t - \Lambda_1(\gamma) \mathbb{E} \| y_{t+1} - y_t \|^2 - C_1 \mathbb{E} \| y_t - y_{t-1}\|^2.
\end{equation}
When $\Lambda_1(\gamma) > 0$, \eqref{eq33} shows that $\Psi_t$ is a decreasing function. Further, according to Lemma \ref{super}, we know
\begin{equation*}
\sum_{t=0}^{\infty} \| y_{t+1} - y_t \|^2 < \infty  \quad a.s.
\end{equation*}
Again, according to Lemma \ref{est-yz} and \eqref{algx} we have 
\begin{equation*}
\sum_{t=0}^{\infty} \| x_{t+1} - x_t \|^2 < \infty  \quad a.s. \quad \textmd{and} \quad \sum_{t=0}^{\infty} \| z_{t+1} - y_{t+1} \|^2 < \infty  \quad a.s..
\end{equation*}
\end{proof}

In the above, we have shown that the sum of squares of the sequence norm has a finite length almost surely. However, we need to further demonstrate that the sequence norm is summable. To do so, we first provide a bound on the norm of subgradient $\partial \Phi(X_t)$.
%%%%%%%%%%%%%%%%%%%%%%%%%%%%%%%%%%%%%%%%%%%%%%%%%%%%%%%%%%%%
\begin{lemma}\label{Thm2}
Suppose that Assumption \ref{ass-fun} holds and $H$ is twice continuously differentiable with a bounded Hessian, i.e., there exists $M > 0$ such that $\| \nabla^2 H(y) \| < M$ for any $y$. Let  $\{(x_t,  y_t, z_t) \}_{t \geq 0}$ be a sequence generated by Algorithm \ref{alg:STOS} using variance-reduced gradient estimators.  Denote $X_t : = (x_t, y_t, z_t, y_{t-1})$. Then, for any $t \geq 1$, we have 
\begin{equation*}
\mathbb{E} dist(0, \partial \Phi(X_t)) \leq C \big(\mathbb{E} \| y_{t+1} - y_t \| + \mathbb{E} \| y_t - y_{t-1} \|\big)  + \mathbb{E} \Gamma_t,
\end{equation*}
where $C = \max\{(M+1 + \frac{\beta}{\gamma})(1+ \gamma L) + V_2, 2C_1 + V_2 \}$. Moreover, for any critical point $X_*= (x_*, y_*, z_*, y_*)$, we have $\mathbb{E} dist(0, \partial \Phi(X_*)) = 0$.
\end{lemma}
\begin{proof}
 First, we denote
\begin{equation*}
\begin{aligned}
\xi_x^t &= \frac{1}{\gamma}(z_t - y_t),\\
\xi_y^t &= \frac{1}{\gamma}(x_{t-1} - x_t) + \frac{1}{\gamma}(z_t - y_t) + \langle \nabla^2 H(y_t), z_t -  y_t \rangle,\\
\xi_z^t &= (1 + \frac{2}{\gamma}) (y_t - z_t) + \frac{1}{\gamma}(x_t - x_{t-1}) + \nabla H(y_t) - \wtilde H(y_t),\\
\xi_{y^{\prime}}^t &= - 2C_1 \| y_t - y_{t-1} \|.
\end{aligned}
\end{equation*}
According to the definition of $\Phi$ in \eqref{Phi},  we can compute 
\begin{subequations}\nonumber
\begin{align}
&\partial_x \Phi (X_t) =  -\frac{1}{\gamma}(y_t - x_t) + \frac{1}{\gamma}(z_t -  x_t) = \frac{1}{\gamma}(z_t - y_t),\\
&\partial_z \Phi (X_t)= \partial G(z_t) - \frac{1}{\gamma}(z_t -  x_t) + \wtilde H(y_t) - (z_t - y_t),\\
&\partial_{y^{\prime}} \Phi (X_t) = -2 C_1 (y_t - y_{t-1}),
\end{align}
\end{subequations}
and 
\begin{equation*}
\begin{aligned}
\partial_y \Phi(X_t) &= \nabla F(y_t) + \nabla H(y_t) + \frac{1}{\gamma}(y_t - x_t) + (z_t - y_t) + \langle \nabla^2 H(y_t), z_t - y_t \rangle - \nabla H(y_t)\\
&= \nabla F(y_t) + \frac{1}{\gamma}(z_t - x_t) + \langle \nabla^2 H(y_t),  z_t - y_t \rangle. 
\end{aligned}
\end{equation*}
Since $\frac{1}{\gamma}(2y_t - z_t - \gamma \wtilde H(y_t) -  x_{t-1}) \in \partial G(z_t)$, we have
\begin{equation*}
(1 + \frac{2}{\gamma}) (y_t - z_t) + \frac{1}{\gamma}(x_t - x_{t-1}) + \nabla H(y_t) - \wtilde H(y_t) \in \partial_z \Phi (X_t).
\end{equation*}
By the optimality condition $0 = \nabla F(y_t) + \frac{1}{\gamma}(y_t - x_{t-1})$, we know
\begin{equation*}
\begin{aligned}
\partial_y \Phi(X_t) &= \frac{1}{\gamma}(x_{t-1} - x_t) + \frac{1}{\gamma}(z_t - y_t) + \langle \nabla^2 H(y_t), z_t -  y_t \rangle.
\end{aligned}
\end{equation*}
Then, we have $(\xi_x^t, \xi_y^t, \xi_z^t, \xi_{y^{\prime}}^t) \in \partial \Phi (X_t)$. By further estimates, we get
\begin{subequations}\nonumber
\begin{align}
\| \xi_x^t \| &= \frac{1}{\gamma}\|z_t - y_t \| = \frac{1}{\gamma} \| x_t - x_{t-1} \| \leq \frac{1+\gamma L}{\gamma} \| y_{t+1} - y_t \|, \\
\| \xi_y^t \| &\leq \frac{1}{\gamma}\|x_t - x_{t-1} \| + \frac{1}{\gamma} \| z_t - y_t \| + \| \nabla^2 H(y_t) \| \| z_t - y_t\|, \\
&\leq (\frac{2}{\gamma} + M )\|x_t - x_{t-1}\|  \leq (\frac{2}{\gamma} + M ) (1+\gamma L) \|y_{t+1} - y_{t}\|,\\
\| \xi_z^t\| &\leq (1+ \frac{2}{\gamma}) \| z_t - y_t \| + \frac{1}{\gamma} \| x_t - x_{t-1} \| + \| \nabla H(y_t) - \wtilde H(y_t) \|, \\
&\leq (1+ \frac{3}{\gamma}) \| x_t -  x_{t-1} \| + \| \nabla H(y_t) - \wtilde H(y_t) \|, \\
&\leq (1+ \frac{3}{\gamma})(1+ \gamma L) \| y_{t+1} - y_t \| + \| \nabla H(y_t) - \wtilde H(y_t) \|,\\
\| \xi_{y^{\prime}}^t \| &= 2C_1 \| y_t - y_{t-1} \|.
\end{align}
\end{subequations}
Then we have the following estimate
\begin{equation*}
\begin{aligned}
&\mathbb{E}_t \| (\xi_x^t, \xi_y^t, \xi_z^t, \xi_{y^{\prime}}^t) \|\\
&\leq (M + 1 + \frac{\beta}{\gamma})(1+ \gamma L) \mathbb{E}_t \| y_{t+1} - y_t \| + 2C_1 \| y_t - y_{t-1} \| + \mathbb{E}_t \| \nabla H(y_t) - \wtilde H(y_t) \|\\
&\leq \Big( (M + 1 + \frac{\beta}{\gamma})(1+ \gamma L) + V_2 \Big) \mathbb{E}_t \| y_{t+1} - y_t \| + (2C_1 + V_2) \| y_t - y_{t-1} \| + \Gamma_t\\
&\leq C \big[ \mathbb{E}_t \| y_{t+1} - y_t \| + \| y_t - y_{t-1} \| \big] + \Gamma_t,
\end{aligned}
\end{equation*}
where $C = \max\{(M+1 + \frac{\beta}{\gamma})(1+ \gamma L) + V_2, 2C_1 + V_2 \}.$  Because $\dist(0, \partial \Phi(X_t)) \leq \| ( \xi_x^t, \xi_y^t, \xi_z^t, \xi_{y^{\prime}}^t) \|$, we get
\begin{equation*}
\mathbb{E} \dist(0, \partial \Phi(X_t)) \leq C \big( \mathbb{E} \| y_{t+1} - y_t \| + \mathbb{E} \| y_t - y_{t-1} \| \big) + \mathbb{E} \Gamma_t.
\end{equation*}
Furthermore, by the Convergence of Estimator, we have $\lim_{t \to \infty} \mathbb{E} \| (\xi_x^t, \xi_y^t, \xi_z^t, \xi_{y^{\prime}}^t) \| = 0$. Then
\begin{equation}\label{eq38}
\lim_{t \to \infty} \mathbb{E} \dist(0, \partial \Phi(X_t)) = 0.
\end{equation}
Suppose that $X_* = (x_*, y_*, z_*, y_*)$ is any cluster point of $X_t := \{x_t, y_t, z_t, y_{t-1} \}_{t\geq 1}$. Then there exists a subsequence $\{ X_{t_k} \}$ satisfying $\lim_{k \to \infty} X_{t_k} = X_*$. By the subproblem \eqref{algz} in Algorithm \ref{alg:STOS}, for any $k \geq 1$, we have
\begin{equation*}
\begin{aligned}
&G(z_{t_k}) + \frac{1}{2\gamma} \| z_{t_k} - (2 y_{t_k} - \gamma \wtilde H(y_{t_k}) - x_{t_k - 1})\|^2\\
&\leq G(z_*) + \frac{1}{2\gamma} \| z_* - (2 y_{t_k} - \gamma \wtilde H(y_{t_k}) - x_{t_k -1}) \|^2.
\end{aligned}
\end{equation*}
This is equivalent to
\begin{equation*}
G(z_{t_k}) + \frac{1}{2\gamma} \Big( \| z_{t_k} \|^2 - \| z_* \|^2 - \langle z_{t_k} - z_*,  2 y_{t_k} - x_{t_k -1} - \gamma \wtilde H(y_{t_k})\rangle \Big) \leq G(z_*),
\end{equation*}
which means $\lim \sup_{k \to \infty} G(z_{t_k}) = G(z_*)$. Because $G$ is a lower semi-continuous function, we know $\lim\inf_{k \to \infty}G(z_{t_k}) \geq G(z_*)$, thus $\lim_{k \to \infty} G(z_{t_k}) = G(z_*)$. Further, by the continuity of $F$ and $H$, we get
\begin{equation*}
\lim_{k \to \infty} \Phi(x_{t_k}, y_{t_k}, z_{t_k}, y_{t_k -1}) = \Phi(x_*, y_*, z_*, y_*) = \Phi (X_*).
\end{equation*}
Combined with \eqref{eq38}, we have
\begin{equation*}
\mathbb{E} \dist(0, \partial \Phi(X_*)) = 0.
\end{equation*} 
\end{proof}

Next, we present the random version of the KL property, which is taken from Lemma 4.5 in \cite{DTLDS}.

\begin{theorem}\label{KL}
Let $\{ X_t \}_{t=0}^{\infty}$ be a bounded sequence of iterates of Algorithm \ref{alg:STOS} using a variance-reduced gradient estimator, and suppose that $X_t$ is not a critical point after a finite number of iterations. Let $\Phi$ be a semialgebraic function satisfying the Kurdyka-$\L$ojasiewicz property (see Definition \ref{KL1}) with exponent $\theta$. Then there exists an index $m$ and desingularizing function $\varphi = a r^{1-\theta}$ so that the following holds almost surely:
\begin{equation}\nonumber
\varphi^{\prime} (\mathbb{E}[\Phi(X_t) - \Phi_{t}^{*}]) \mathbb{E}\dist(0, \partial \Phi(X_t)) \geq 1,~~~~~\forall t > m,
\end{equation}
where $\Phi_{t}^{*}$ is an non-decreasing sequence converging to $\mathbb{E}\Phi(X_*)$ for some $X_* \in \Omega$, where $\Omega$ is the set of cluster points of $\{ X_t \}_{t \geq 0}$.
\end{theorem}
In the following, we provide the convergence for STOS (Algorithm \ref{alg:STOS}) under the assumption that the objective functions $F, G$, and $H$ are semi-algebraic.

\begin{proof}[Proof of  Theorem \ref{thmKLconver}]
If $\theta \in (0, \frac{1}{2})$, then $\Phi$ also satisfies the KL property with exponent $\frac{1}{2}$, so we only consider the case $\theta \in [\frac{1}{2}, 1).$ By Theorem \ref{KL} there exists a function $\varphi_0(r) = ar^{1-\theta}$ such that, almost surely,
\begin{equation*}
\varphi_0^{\prime} (\mathbb{E}[\Phi(X_t) - \Phi_t^*]) \mathbb{E}\dist(0, \partial\Phi(X_t)) \geq 1, ~~\forall t >m.
\end{equation*}
Using the bound on $\mathbb{E}\dist(0, \partial\Phi(X_t))$ in Theorem  \ref{Thm2} and Jensen$^{\prime}$s inequality, we have
\begin{equation}\label{eq47}
\begin{aligned}
\mathbb{E} \dist(0, \partial \Phi(x_t)) &\leq  C \left( \mathbb{E} \| y_{t+1} - y_t\|  + \mathbb{E} \| y_t - y_{t-1} \| \right)+ \mathbb{E} \Gamma_t\\
&\leq C \left( \sqrt{\mathbb{E} \| y_{t+1} - y_t \|^2 } + \sqrt{\mathbb{E} \| y_t - y_{t-1} \|^2} \right) + \sqrt{s\mathbb{E} \Upsilon_t}.
\end{aligned}
\end{equation}
Because $\Gamma_t = \sum_{i=1}^{s} \| v_t^i \|$ for some vectors $v_t^i$, we have $\mathbb{E} \Gamma_t = \mathbb{E} \sum_{i=1}^s \| v_t^i \| \leq \mathbb{E} \sqrt{s\sum_{i=1}^s \| v_t^i \|^2} \leq \sqrt{s\mathbb{E}\Upsilon_t}$. By the geometric decay in Definition \ref{var-redu} we can bound the term $\sqrt{\mathbb{E}\Upsilon_t}$:
\begin{equation*}
\begin{aligned}
\sqrt{\mathbb{E}\Upsilon_{t+1}} &\leq \sqrt{(1 - \rho) \mathbb{E}\Upsilon_{t} + V_{\Upsilon} \mathbb{E} \| y_{t+1} - y_{t} \|^2 + V_{\Upsilon} \mathbb{E} \| y_{t} - y_{t-1} \|^2}\\
&\leq \sqrt{1-\rho} \sqrt{\mathbb{E}\Upsilon_{t}} + \sqrt{V_{\Upsilon}} \sqrt{\mathbb{E} \| y_{t+1} -y_{t} \|^2} + \sqrt{V_{\Upsilon}} \sqrt{\mathbb{E} \| y_{t} -y_{t-1} \|^2}\\
&\leq (1 - \frac{\rho}{2}) \sqrt{\mathbb{E}\Upsilon_{t}} + \sqrt{V_{\Upsilon}} \sqrt{\mathbb{E} \| y_{t+1} -y_{t} \|^2} + \sqrt{V_{\Upsilon}} \sqrt{\mathbb{E} \| y_{t} -y_{t-1} \|^2}.
\end{aligned}
\end{equation*}
Rearranging  and multiplying by $\sqrt{s}$ yields
\begin{equation}\label{eq50}
\begin{aligned}
\sqrt{s\mathbb{E}\Upsilon_{t}} &\leq \frac{2\sqrt{s}}{\rho} (\sqrt{\mathbb{E}\Upsilon_{t}} - \sqrt{\mathbb{E}\Upsilon_{t+1}}) + \frac{2\sqrt{s}\sqrt{V_{\Upsilon}}}{\rho} \sqrt{\mathbb{E}\| y_{t+1} - y_{t} \|^2} +  \frac{2\sqrt{s}\sqrt{V_{\Upsilon}}}{\rho} \sqrt{\mathbb{E}\| y_{t} - y_{t-1} \|^2}.
\end{aligned}
\end{equation}
Substituting \eqref{eq50} into \eqref{eq47}, we get 
\begin{equation}\label{eq51}
\begin{aligned}
&\mathbb{E} \dist(0, \partial \Phi(X_t))\\
&\leq C\sqrt{\mathbb{E}\| y_{t+1} - y_t \|^2} + C \sqrt{\mathbb{E} \| y_t - y_{t-1} \|^2} + \frac{2\sqrt{s}}{\rho} (\sqrt{\mathbb{E}\Upsilon_t} - \sqrt{\mathbb{E} \Upsilon_{t+1}}) + \frac{2\sqrt{s V_{\Upsilon}}}{\rho} \sqrt{\mathbb{E}\| y_{t+1}- y_t \|^2} \\
&\quad + \frac{2\sqrt{s V_{\Upsilon}}}{\rho} \sqrt{\mathbb{E}\| y_{t}- y_{t-1} \|^2}\\
&\leq (C + \frac{2\sqrt{s V_{\Upsilon}}}{\rho})  \sqrt{\mathbb{E}\| y_{t+1}- y_t \|^2} + (C + \frac{2\sqrt{s V_{\Upsilon}}}{\rho})  \sqrt{\mathbb{E}\| y_{t}- y_{t-1} \|^2} + \frac{2\sqrt{s}}{\rho} (\sqrt{\mathbb{E}\Upsilon_t} - \sqrt{\mathbb{E} \Upsilon_{t+1}})\\
&\leq (C + \frac{2\sqrt{s V_{\Upsilon}}}{\rho}) ( \sqrt{\mathbb{E}\| y_{t+1}- y_t \|^2} + \sqrt{\mathbb{E}\| y_{t}- y_{t-1} \|^2} ) + \frac{2\sqrt{s}}{\rho}(\sqrt{\mathbb{E}\Upsilon_t} - \sqrt{\mathbb{E} \Upsilon_{t+1}}).
\end{aligned}
\end{equation}
Define $C_t = (C + \frac{2\sqrt{s V_{\Upsilon}}}{\rho}) ( \sqrt{\mathbb{E}\| y_{t+1}- y_t \|^2} + \sqrt{\mathbb{E}\| y_{t}- y_{t-1} \|^2} ) + \frac{2\sqrt{s}}{\rho}(\sqrt{\mathbb{E}\Upsilon_t} - \sqrt{\mathbb{E} \Upsilon_{t+1}})$. Then we have $\mathbb{E}\dist(0, \partial \Phi(x_t)) \leq C_t$. Therefore, $\varphi_0^{\prime} (\mathbb{E}[\Phi(x_t) -  \Phi(x_t^*)]) C_t \geq 1$ for $t > m.$ Using the definition of $\varphi_0$, we can rewrite this inequality as
\begin{equation*}
\frac{a(1- \theta){\bf C_t}}{(\mathbb{E}[\Phi(X_t) - \Phi_t^*])^{\theta}} \geq 1,~~~~~\forall t >m.
\end{equation*}
Next we prove that the above inequality holds for $\Psi_t$, for which an additional term $\mathcal{O} \{  (\mathbb{E} [ \| \wtilde  H(y_{t})$ $- \nabla H(y_{t})\|^2 + \Upsilon_t + \| y_t - y_{t-1} \|^2 ])^\theta \}$ needs to be added to the denominator. By ${\bf C_t} \geq \mathcal{O} (\sqrt{\mathbb{E}\| y_{t+1} - y_{t} \|^2} $ $+ \sqrt{\mathbb{E}\| y_{t} - y_{t-1} \|^2} + \sqrt{\mathbb{E} \Upsilon_{t-1}})$, $\mathbb{E}\| \wtilde  H(y_{t}) - \nabla H(y_{t})\|^2 \to 0,~\mathbb{E}\Upsilon_t \to  0$,~$\mathbb{E}\| y_t - y_{t-1}\| \to 0$ and $\theta > \frac{1}{2}$, there exists an index $m$ and a constant $c> 0$ such that
\begin{equation*}
\begin{aligned}
&  \mathbb{E} \big[ - \| \wtilde  H(y_{t}) - \nabla H(y_{t})\|^2 +  \frac{5}{2\rho} \Upsilon_t + \frac{5V_{\Upsilon} + 5V_1 \rho}{2\rho}\| y_t - y_{t-1}\|^2 \big] \\
&\leq  \mathbb{E} \big[ \| \wtilde H(y_t) - \nabla H(y_t) \|^2 + \frac{5}{2\rho} \Upsilon_t + \frac{5V_{\Upsilon} + 5V_1 \rho}{2\rho} \| y_t - y_{t-1} \|^2 \big]\\
&\leq \mathbb{E} \big[ \Upsilon_t + V_1 \| y_{t+1} - y_t \|^2 + V_1 \| y_t - y_{t-1} \|^2 + \frac{5}{2\rho} \Upsilon_t + \frac{5V_{\Upsilon} + 5V_1 \rho}{2\rho} \| y_t - y_{t-1}\|^2 \big]\\
&\leq \mathbb{E} \big[ (1+ \frac{5}{2\rho}) \Upsilon_t + V_1 \| y_{t+1} - y_t \|^2 + \left( V_1 + \frac{5V_{\Upsilon} + 5V_1 \rho}{2\rho} \right) \| y_t - y_{t-1} \|^2  \big]\\
&\leq \mathcal{O}  \big( \mathbb{E} [\Upsilon_{t} + \| y_{t+1} - y_{t} \|^2 + \| y_{t} - y_{t-1} \|^2] \big)\\
&\leq c {\bf C_t}^{\frac{1}{\theta}},~~~\forall t > m.
\end{aligned}
\end{equation*}
Denote 
\begin{equation*}
\Lambda_t \overset{\mathrm{def}}{=} -\| \wtilde  H(y_{t}) - \nabla H(y_{t})\|^2 +  \frac{5}{2\rho} \Upsilon_t + \frac{5V_{\Upsilon} + 5V_1 \rho}{2\rho}\| y_t - y_{t-1}\|^2. 
\end{equation*}
Since the terms above are small compared to ${\bf C_t}$, we can find a constant $d \in (c, +\infty)$ such that
\begin{equation*}
\begin{aligned}
\frac{ad(1-\theta) {\bf C_t}}{(\mathbb{E}[\Phi(X_t) - \Phi_t^*])^{\theta} + \left(  \mathbb{E} \Lambda_t \right)^{\theta}} \geq 1
\end{aligned}
\end{equation*}
for all $t > m$. Using the fact that $(a + b)^{\theta} \leq a^{\theta} + b^{\theta}$, we get
\begin{equation*}
\begin{aligned}
\frac{ad(1 - \theta){\bf C_t}}{(\mathbb{E}[\Psi_t - \Phi_t^*])^{\theta}} 
= \frac{ad(1 - \theta){\bf C_t}}{(\mathbb{E}[\Phi_{t} - \Phi_t^* + \Lambda_t ])^{\theta}}
\geq \frac{ad(1 - \theta){\bf C_t}}{(\mathbb{E}[\Phi_{t} - \Phi_t^*])^\theta + (\mathbb{E} \Lambda_t)^{\theta}} 
\geq 1,~~\forall t >m.
\end{aligned}
\end{equation*}
Hence, for $\varphi(r) = adr^{1-\theta}$,
\begin{equation*}
\varphi^{\prime}(\mathbb{E}[ \Psi_t - \Phi_t^*] ) {\bf C_t} \geq 1, ~~~\forall t >m.
\end{equation*}
By the concavity of $\varphi$,
\begin{equation*}
\begin{aligned}
\varphi(\mathbb{E} [\Psi_t - \Phi_t^*]) - \varphi(\mathbb{E}[\Psi_{t+1} - \Phi_{t+1}^*])& \geq \varphi^{\prime} (\mathbb{E}[\Psi_t - \Phi_t^*]) (\mathbb{E} [\Psi_t - \Phi_t^* + \Phi_{t+1}^* - \Psi_{t+1}])\\
&\geq \varphi^{\prime}(\mathbb{E} [\Psi_t - \Phi_t^*]) (\mathbb{E}[\Psi_t - \Psi_{t+1}]),
\end{aligned}
\end{equation*}
where the last inequality is due to the non-decreasing property of $\Phi_t^*$. Denote $\Delta_{p,q} \overset{\mathrm{def}}{=} \varphi(\mathbb{E}[\Psi_p - \Phi_p^*]) - \varphi (\mathbb{E}[\Psi_q - \Phi_q^*])$, then
\begin{equation*}
\Delta_{t,t+1}{\bf C_t} \geq \mathbb{E}[\Psi_t - \Psi_{t+1}].
\end{equation*}
Since $\mathbb{E}[\Psi_t - \Psi_{t+1}] \geq \Lambda(\gamma) \mathbb{E} \| y_{t+1} - y_t \|^2 + C_1 \mathbb{E}\| y_t - y_{t-1}\|^2$, we have
\begin{equation}\label{eq60}
\begin{aligned}
\Delta_{t,t+1} {\bf C_t}&\geq \Lambda(\gamma) \mathbb{E} \| y_{t+1}  - y_t \|^2 + C_1 \mathbb{E} \| y_{t}  - y_{t-1} \|^2\\
&\geq \max\{\Lambda(\gamma),C_1\} \mathbb{E} \| y_{t+1} - y_t \|^2 + \max\{\Lambda(\gamma),C_1\} \mathbb{E} \| y_{t} - y_{t-1} \|^2.
\end{aligned}
\end{equation}
Applying Young's inequality to \eqref{eq60} yields
\begin{equation*}
\begin{aligned}
&2\sqrt{\mathbb{E}\| y_{t+1} - y_t \|^2 + \mathbb{E}\| y_{t} - y_{t-1} \|^2}\\
&\leq 2\sqrt{\big(\max\{\Lambda(\gamma),C_1\}\big)^{-1}{\bf C_t} \Delta_{t,t+1}} \leq \frac{{\bf C_t}}{2\big( C+ \frac{2\sqrt{sV_{\Upsilon}}}{\rho}\big)} + \frac{2\big( C+ \frac{2\sqrt{sV_{\Upsilon}}}{\rho}\big) \Delta_{t,t+1}}{\max\{\Lambda(\gamma),C_1\} }\\
&\leq \frac{\sqrt{\mathbb{E}\| y_{t+1} - y_{t}\|^2}}{2} + \frac{\sqrt{\mathbb{E}\| y_t - y_{t-1} \|^2}}{2} + \frac{\sqrt{s}(\sqrt{\mathbb{E}\Upsilon_{t}} - \sqrt{\mathbb{E}\Upsilon_{t+1}})}{\big( C+ \frac{2\sqrt{sV_{\Upsilon}}}{\rho}\big)\rho} + \frac{2\big( C+ \frac{2\sqrt{sV_{\Upsilon}}}{\rho}\big) \Delta_{t,t+1}}{\max\{\Lambda(\gamma),C_1\} }.\\
\end{aligned}
\end{equation*}
Then we have
\begin{equation}\label{eq62}
\begin{aligned}
&\sqrt{\mathbb{E}\| y_{t+1} - y_t \|^2} + \sqrt{\mathbb{E}\| y_{t} - y_{t-1} \|^2}\\
&\leq \sqrt{2} \sqrt{\mathbb{E}\| y_{t+1} - y_t \|^2 + \mathbb{E}\| y_{t} - y_{t-1} \|^2} \\
&\leq  \frac{\sqrt{2}\sqrt{\mathbb{E}\| y_{t+1} - y_{t}\|^2}}{4} + \frac{\sqrt{2}\sqrt{\mathbb{E}\| y_t - y_{t-1} \|^2}}{4} + \frac{\sqrt{2s}(\sqrt{\mathbb{E}\Upsilon_{t}} - \sqrt{\mathbb{E}\Upsilon_{t+1}})}{2\big( C+ \frac{2\sqrt{sV_{\Upsilon}}}{\rho}\big)\rho} + \frac{2\big( C+ \frac{2\sqrt{sV_{\Upsilon}}}{\rho}\big) \Delta_{t,t+1}}{\max\{\Lambda(\gamma),C_1\} }.\\
\end{aligned}
\end{equation}
Summing inequality \eqref{eq62} from $t=m$ to $i$,
\begin{equation*}
\begin{aligned}
&\sum_{t=m}^{i} \sqrt{\mathbb{E} \| y_{t+1} - y_t \|^2} + \sum_{t=m}^{i} \sqrt{\mathbb{E} \| y_{t} - y_{t-1} \|^2} \\
&\leq \frac{\sqrt{2}}{4} \sum_{t=m}^{i} \sqrt{\mathbb{E} \| y_{t+1} - y_{t} \|^2}  + \frac{\sqrt{2}}{4} \sum_{t=m}^{i} \sqrt{\mathbb{E} \| y_{t} - y_{t-1} \|^2} + \frac{\sqrt{2s}}{2\big( C+ \frac{2\sqrt{sV_{\Upsilon}}}{\rho}\big) \rho} \sum_{t=m}^i (\sqrt{\mathbb{E} \Upsilon_{t} } - \sqrt{\mathbb{E}\Upsilon_{t+1}})\\
&\quad + \frac{\sqrt{2} \big( C+ \frac{2\sqrt{sV_{\Upsilon}}}{\rho}\big)}{\max\{\Lambda(\gamma),C_1\}} \sum_{t=m}^{i} \Delta_{t, t+1}\\
&\leq \frac{\sqrt{2}}{4} \sum_{t=m}^{i} \sqrt{\mathbb{E} \| y_{t+1} - y_{t} \|^2} + \frac{\sqrt{2}}{4} \sum_{t=m}^{i} \sqrt{\mathbb{E} \| y_{t} - y_{t-1} \|^2} + \frac{\sqrt{2s}}{2 \big( C+ \frac{2\sqrt{sV_{\Upsilon}}}{\rho}\big)\rho} (\sqrt{\mathbb{E} \Upsilon_{m} } - \sqrt{\mathbb{E}\Upsilon_i}) \\
&\quad + \frac{\sqrt{2}  \big( C+ \frac{2\sqrt{sV_{\Upsilon}}}{\rho}\big)}{\max\{\Lambda(\gamma),C_1\}} \Delta_{m, i+1}.\\
\end{aligned}
\end{equation*}
This implies that
\begin{equation*}
\begin{aligned}
&\sum_{t=m}^{i} \sqrt{\mathbb{E} \| y_{t+1} - x_t \|^2} + \sum_{t=m}^{i} \sqrt{\mathbb{E} \| y_{t} - y_{t-1} \|^2} \\
&\leq \frac{(2\sqrt{2}+1)\sqrt{s}}{3 \big( C+ \frac{2\sqrt{sV_{\Upsilon}}}{\rho}\big)\rho} \sqrt{\mathbb{E} \Upsilon_{m} }  + \frac{(4\sqrt{2} + 2)\big( C+ \frac{2\sqrt{sV_{\Upsilon}}}{\rho}\big)}{3\max\{\Lambda(\gamma),C_1\}} \Delta_{m, i+1}\\ 
&\leq C_2( \sqrt{\mathbb{E} \Upsilon_{m} } +  \Delta_{m, i+1} ),
\end{aligned}
\end{equation*}
where $C_2 = \max\{\frac{(2\sqrt{2}+1)\sqrt{s}}{3\big( C+ \frac{2\sqrt{sV_{\Upsilon}}}{\rho}\big)\rho}, \frac{(4\sqrt{2} + 2)\big( C+ \frac{2\sqrt{sV_{\Upsilon}}}{\rho}\big)}{3\max\{\Lambda(\gamma),C_1\}} \}$. Using Jensen's inequality yields
\begin{equation*}
\begin{aligned}
&\sum_{t=m}^{i} \mathbb{E} \| y_{t+1} - y_t \| + \sum_{t=m}^{i} \mathbb{E} \| y_{t} - y_{t-1} \| \\
&\leq \sum_{t=m}^{i} \sqrt{\mathbb{E} \| y_{t+1} - y_t \|^2} + \sum_{t=m}^{i} \sqrt{\mathbb{E} \| y_{t} - y_{t-1} \|^2}\\
&\leq C_2 ( \sqrt{\mathbb{E} \Upsilon_{m-1} } +  \Delta_{m, i+1} ).
\end{aligned}
\end{equation*}
Since $\Delta_{m, i+1}$ is bounded, we obtain
\begin{equation*}
\sum_{t=m}^{\infty} \mathbb{E} \| y_{t+1} - y_t \| < \infty~~ \textmd{and}~~\sum_{t=m}^{\infty} \mathbb{E} \| x_{t+1} - x_t \| < \infty.
\end{equation*}
Together with \eqref{algx}, we get
\begin{equation*}
\sum_{t=m}^{\infty} \mathbb{E} \| z_{t+1} - z_t \| < \infty.
\end{equation*}
\end{proof}

\end{document}